\documentclass[11pt]{amsart}

\usepackage{microtype}

\swapnumbers

\usepackage{amssymb,amsfonts}
\usepackage[mathscr]{euscript}
\usepackage[alphabetic]{amsrefs}
\usepackage{bbm}  

\usepackage[ps,matrix,arrow,curve,cmtip]{xy}

\title[Power Operations are Koszul]{Rings of Power Operations for
  Morava $E$-theories are 
  Koszul}
\author{Charles Rezk}
\date{ \today}
\address{Department of Mathematics \\
University of Illinois at Urbana-Champaign \\ 
Urbana, IL}
\email{rezk@math.uiuc.edu}

\numberwithin{equation}{section}

\makeatletter
  \let\c@subsection\c@equation
\makeatother

\theoremstyle{plain}   

\newtheorem{prop}[subsection]{Proposition}
\newtheorem{cor}[subsection]{Corollary}
\newtheorem{lemma}[subsection]{Lemma}

\newtheorem*{main-theorem}{Main Theorem}

\theoremstyle{remark}
\newtheorem{rem}[subsection]{Remark}    
\newtheorem{exam}[subsection]{Example}

\theoremstyle{plain}

\DeclareMathOperator{\id}{id}

\DeclareMathOperator{\Cok}{Cok}

\DeclareMathOperator{\Ker}{Ker}

\newcommand{\End}{{\operatorname{End}}}
\newcommand{\Hom}{{\operatorname{Hom}}}

\newcommand{\ra}{\rightarrow}

\newcommand{\xra}{\xrightarrow}

\DeclareMathOperator{\Ext}{Ext}
\DeclareMathOperator{\Tor}{Tor}

\newcommand{\len}[1]{\lvert#1\rvert}

\newcommand{\powser}[1]{[\![#1]\!]}

\newcommand{\F}{\mathbb{F}}
\newcommand{\Z}{\mathbb{Z}}

\newcommand{\R}{\mathbb{R}}
\newcommand{\Q}{\mathbb{Q}}
\newcommand{\C}{\mathbb{C}}

\newcommand{\sm}{\wedge} 

\newcommand{\dfn}{\textbf}

\def\defeq{\overset{\mathrm{def}}=}

\begin{document}

\newcommand{\Spaces}{\mathrm{Spaces}}
\newcommand{\Spectra}{\mathrm{Spectra}}
\newcommand{\hSpectra}{h\mathrm{Spectra}}

\newcommand{\Ab}{\mathrm{Ab}}
\newcommand{\Set}{\mathrm{Set}}

\newcommand{\length}{\operatorname{length}}
\newcommand{\mesh}{\operatorname{mesh}}

\newcommand{\eet}{{\perp}}
\newcommand{\lin}{\mathcal{L}}

\newcommand{\minel}{\underline{0}}
\newcommand{\maxel}{\underline{1}}

\newcommand{\B}{\mathcal{B}}
\newcommand{\rB}{\overline{\B}}
\newcommand{\hatB}{\hat{\B}}
\newcommand{\chB}{\check{\B}}
\newcommand{\ddB}{\ddot{\B}}

\newcommand{\wt}[1]{\widetilde{#1}}
\newcommand{\ol}[1]{\overline{#1}}
\newcommand{\ul}[1]{\underline{#1}}
\newcommand{\mr}[1]{\mathrm{#1}}
\newcommand{\wh}[1]{\widehat{#1}}
\newcommand{\mc}[1]{\mathcal{#1}}

\newcommand{\bfone}{\mathbbm{1}}

\newcommand{\rank}{\operatorname{rank}}

\newcommand{\cE}[1]{E^{\sm}_{#1}}

\newcommand{\Mod}[1]{\mathrm{Mod}_{#1}}
\newcommand{\Alg}[1]{{\mathrm{Alg}_{#1}}}

\newcommand{\algapprox}{\mathbb{T}}
\newcommand{\talgapprox}{\widetilde{\mathbb{T}}}
\newcommand{\Pfree}{\mathbb{P}}
\newcommand{\gralgcat}{{\mathrm{Alg}^*_{\algapprox}}}

\newcommand{\Si}{\Sigma^\infty}
\newcommand{\Sip}{\Sigma^\infty_+}

\newcommand{\Map}{\operatorname{Map}}

\newcommand{\LKan}{\operatorname{LKan}}

\begin{abstract}
We show that the ring of power operations for any Morava $E$-theory
is Koszul.  
\end{abstract}

\maketitle


\section{Introduction}

\subsection{Power operations for commutative ring spectra}

Given a structured commutative ring spectrum $E$, it is an important
task to 
understand its theory of \emph{power operations}.  For this paper,
power operations
are \emph{additive} operations on the  homotopy groups of commutative
$E$-algebras 
which arise as the residue of the (non-linear) multiplicative
structure of a 
structured commutative ring spectrum.  The closest classical analogue
to power operations is the  Frobenius map on commutative rings of finite
characteristic; indeed there is a close connection between power
operations and the  Frobenius map.  The most familiar examples of power
operations are the Steenrod operations on the mod $p$ homology of a
space, which are in fact a specialization of more general operations
defined on the homotopy of a commutative $H\F_p$-algebra, where
$H\F_p$ denotes the mod $p$ Eilenberg-Mac Lane spectrum.  

\subsection{Power operations for Morava $E$-theory}

In this paper, we address the ring of power operations for a Morava
$E$-theory spectrum.  Recall that to each formal group $G$ of height
$n$ defined over a perfect field $k$ of characteristic $p$, there is
an associated \dfn{Morava $E$-theory spectrum}, an even periodic
complex orientable theory whose associated formal group is the
\emph{universal deformation} of $G$ (as constructed by Lubin and
Tate).  The coefficient ring of $E$ has the form
\[
\pi_*E = \mathbb{W}k\powser{u_1,\dots,u_{n-1}}[u^{\pm}]
\]
where $\mathbb{W}k$ is the unramified extension of $\Z_p$ with residue
field $k$, $u_1,\dots,u_{n-1}\in \pi_0E$, and $u\in \pi_2E$.  Note
that $\pi_0E$ is a complete local ring.
By a theorem of Goerss, Hopkins, and Miller
\cite{goerss-hopkins-moduli-spaces}, every Morava $E$-theory spectrum
admits the structure of a commutative $S$-algebra in an essentially
unique way.  

Fix  $E$ to be a Morava $E$-theory spectrum.  Let $\Alg{E}$ denote the
category of commutative $E$-algebras, and let $\Alg{E,K(n)}$ denote
the category of \emph{$K(n)$-local} commutative $E$-algebras.  The
\dfn{ring of (additive) $E$-theory power operations} $\Gamma$ is a
set of natural additive operations on  
$\pi_0$ of $K(n)$-local commutative $E$-algebras.  In fact, the
completion $\Gamma^{\sm}_{\mathfrak{m}}$ of $\Gamma$ at the maximal
ideal of $\pi_0E$ is precisely the 
endomorphism ring of the homotopy functor $\pi_0\colon h\Alg{E,K(n)}\ra \Ab$.
(We give the definition of $\Gamma$ in 
\ref{subsec:rings-gamma-q}, and describe its relation to additive power
operations in \eqref{subsec:aside-on-additive-ops}.)

In this paper, we will need to consider a variant of  $\Gamma$.  Given
an \emph{augmented} $K(n)$-local commutative $E$-algebra, we may
consider a set $\Delta$ of natural additive operations on the
\emph{indecomposable quotient} of $\pi_0$.  The completion
$\Delta^{\sm}_{\mathfrak{m}}$ at the maximal ideal of $\pi_0E$ is the 
endomorphism ring of the indecomposables functor $Q(\pi_0)\colon
\Alg{E,K(n)}_{/E}\ra \Ab$, i.e., it is a ring of operations acting on the ``cotangent space'' of
the homotopy of a commutative $E$-algebra at the augmentation.   There
is an isomorphism of rings $\Gamma\approx \Delta$
\eqref{cor:gamma-delta-iso-free}.

\begin{rem}
If $X$ is any space, then the spectrum $E^{X_+}$ of functions from the
suspension spectrum of $X$ to $E$ is naturally a $K(n)$-local
$E_\infty$-ring spectrum.  Thus, the cohomology group
$E^0(X)=\pi_0(E^{X_+})$ is naturally equipped with the structure of a
$\Gamma$-module, so $\Gamma$ gives rise to a family of unstable
cohomology operations for $E$.
\end{rem}

\begin{exam}
The simplest example of a Morava $E$-theory is complex $K$-theory
completed at a prime $p$, which is associated to the multiplicative
group law (of height 1) over $\F_p$.  In this case $\Gamma\approx
\Z_p[\psi^p]$, where 
$\psi^p$ is a power operation which on $K$-theory cohomology rings
gives the $p$th Adams operation.

There is a \emph{non-additive} operation $\theta^p$
which acts on $\pi_0$ of a $K(1)$-local commutative $K$-algebra, which
is characterized by the identity $\psi^p(x)=x^p+p\theta^p(x)$.
Although $\theta^p$ is not additive, a consequence of this identity is
that $\theta^p$ is additive \emph{modulo decomposables}.  It turns
out that $\Delta \approx \Z_p[\theta^p]$.
\end{exam}

\begin{rem}\label{rem:gamma-grading}
Explicitly, $\Gamma=\bigoplus_{k\geq0} \Gamma[k]$ is a graded ring
with
\[
\Gamma[k] \approx \Ker \bigl[ \cE{0}B\Sigma_{p^k} \ra
\bigoplus_{0<j<p^k} \cE{0}B(\Sigma_j\times \Sigma_{p^k-j}) \bigr],
\]
the map being induced by the transfer associated to the inclusions
$\Sigma_j\times \Sigma_{p^k-j}\subset \Sigma_{p^k}$ 
\eqref{subsec:rings-gamma-q}.   Here 
$\cE{*}(X)$ denotes the \emph{completed $E$-homology} of $X$
\eqref{subsec:completed-e-homology}.  Similarly,
$\Delta=\bigoplus_{k\geq0} \Delta[k]$ is a graded ring with
\[
\Delta[k] \approx \Cok\bigl[ \bigoplus_{0<j<p^k} \cE{0}
B(\Sigma_j\times \Sigma_{p^k-j}) \ra \cE{0}B\Sigma_{p^k}\bigr],
\]
the map being induced by the evident inclusions of groups
\eqref{subsec:rings-delta-q}. 

A key result, due to Strickland
\cite{strickland-morava-e-theory-of-symmetric} is that each
$\Gamma[k]$ and $\Delta[k]$ is a finitely generated free $E_0$ module
\eqref{prop:gamma-finite-free}.  
\end{rem}

\begin{rem}
For a Morava $E$-theory associated to a height $n$-formal group
$G/k$, its ring $\Gamma$ of power operations is entirely determined
by the formal group $G/k$.  In particular, the  category of modules
over $\Gamma$ can be 
identified with a certain category of quasi-coherent sheaves on a
stack of ``deformations of $G/k$ and isogenies which lift powers of
Frobenius''.  This identification is an observation of Ando, Hopkins,
and Strickland (it is prefigured in
\cite{ando-hopkins-strickland-h-infinity}, 
and an explicit statement
is given \cite{rezk-congruence-condition}*{Thm.\ B}).  We won't need
this identification in this paper.
\end{rem}

The ring $\Gamma$ of \emph{additive} power operations for $E$ is just one
piece of the story of its power operations, which are controlled by a certain
monad $\algapprox$ on the category of $E_*$-modules.  This monad was
introduced in \cite{rezk-congruence-condition}; we will review some of
its properties in \eqref{subsec:algebraic-functor}.

\subsection{Koszul algebras and the main theorem}

The notion of Koszul algebra was introduced by Priddy
\cite{priddy-koszul-resolutions}.  Roughly speaking, a Koszul algebra
is a graded $k$-algebra $A$ which admits a \emph{Koszul complex},
namely a functorial resolution $\mc{C}_\bullet(M) = \bigl(\cdots \ra
A\otimes_k C_n\otimes_k M \ra 
A\otimes_k C_{n-1}\otimes_k M\ra\cdots\bigr)$ of left $A$-modules, which is
``minimal'', in the sense that $k\otimes_A \mc{C}_\bullet(M)$ has
trivial differentials, and thus $C_n=\Tor^A_n(k,M)$ (if $M$ is flat
over $k$).  

Our main result is the following.

\begin{main-theorem}\label{main-theorem}
The rings $\Gamma$ and $\Delta$ of power operations for any Morava
$E$-theory, are 
Koszul with respect to their natural grading \eqref{rem:gamma-grading}.
\end{main-theorem}
A subtlety is that although $\Gamma$ contains the coefficient ring
$E_0=\pi_0 E$ of the theory $E$, this subring is not central in
$\Gamma$.  The notion of 
Koszul we will use (described in \S\ref{sec:koszul-rings}) will make
sense for such rings.   Furthermore, with this definition, it will be
a \emph{consequence} of the main theorem that the ring $\Gamma$ is
\emph{quadratic} \eqref{prop:koszul-implies-quadratic}, i.e., there is
an isomorphism 
$\Gamma \approx T_{E_0}(\Gamma[1])/(R)$ with a quotient of the
tensor algebra on the degree 1 part by relations $R\subseteq
\Gamma[2]$ contained in the degree 2 part.

Furthermore, a straightforward calculation of ranks
\eqref{prop:koszul-vanishing} 
shows that, for a Morava $E$-theory of height $n$, the objects $C_k$
satisfy $C_k=0$ when $k>n$.  That is, the ring of power operations for
Morava $E$-theory of height $n$ has a \emph{bounded} Koszul
resolution, of length $n+1$.

The proof of the main theorem is given in
\eqref{subsec:proof-of-main-thm}.

\subsection{Some consequences of the theorem}

We briefly list some  applications and consequences of this
result.  

\textit{Description of $\gamma$ at height 2.}  As noted, a consequence of our
theorem is that the ring $\Gamma$ is quadratic with respect to the
grading $\Gamma=\bigoplus \Gamma[k]$: i.e., it is generated by
$\Gamma[1]$ with relations given by the kernel of multiplication
$\Gamma[1]\otimes_{E_0} 
\Gamma[1]\ra \Gamma[2]$.  Furthermore $\Gamma[1]$ and $\Gamma[2]$,
by a theorem of Strickland
\cite{strickland-finite-subgroups-of-formal-groups}, are as
$E_0$-modules dual to rings 
which classify subgroup divisors of order $p$ and $p^2$ of the formal
group of $E$.  Thus, an explicit calculation of the structure of
$\Gamma$ can be recovered from an explicit understanding of these
subgroup rings.

As noted above, the height 1 case is well known (at every prime $\Gamma$ is a
polynomial ring on one generator).  An explicit description of
$\Gamma$ at  height 2 by the author at the prime $2$
\cite{rezk-dyer-lashof-example}, and by Yifei Zhu
at primes $3$ and $5$ \cite{zhu-power-ops-at-3}.  Furthermore, Zhu has
described a uniform
procedure for calculating the height 2 algebra at all primes
\cite{zhu-modular-equations-lubin-tate}. 

\textit{Spaces of maps between commutative $E$-algebras.}
Homological algebra associated to $\Gamma$ appears as input to
spectral sequences and obstruction theory which compute maps between
$K(n)$-local commutative ring spectra.  These tools are a special case
of the general machinery of Goerss and Hopkins
\cite{goerss-hopkins-moduli-spaces}.  

Here is an example of how our theorem aids such calculations (details
are in the preprint \cite{rezk-power-ops-structure-calculation}).  Fix
a height $n$ Morava $E$-theory.
Suppose
$R$ and $F$ are $K(n)$-local commutative 
$E$-algebras equipped with an augmentation to $E$.  Then there is a spectral sequence
\[
E_2^{s,t}\Longrightarrow \pi_{t-s}\Map(R,F)
\]
computing homotopy groups of the space of maps of augmented
$E$-algebras.  The $E_2$-term is Quillen cohomology of the 
$\algapprox$-algebra $\pi_*R$ with coefficients in $\pi_*F$.

When $\pi_*R$ is \emph{smooth} as a $\pi_*E$-algebra, a composite
functor spectral sequence an isomorphism of the $E_2$-term with 
\[
E_2^{s,t} = \Ext_{\Delta}^s( \wh{Q}(\pi_*R),
 \pi_* \ol{F}).
\]
This is Ext in graded $\Delta$-modules, where
$\wh{Q}$ denotes the completion of the indecomposables of $\pi_*R$
with respect to the maximal ideal of $E_*$, and $\ol{F}$ is the fiber
of the augmentation $F\ra E$.  

When $\pi_*R$ is smooth over $\pi_*E$, the indecomposables
$\wh{Q}(\pi_*R)$ are projective over $\pi_*E$.  Recall that $\Delta$
(the ring of operations which acts naturally on indecomposables) is
isomorphic to $\Gamma$, and hence is Koszul by \eqref{main-theorem}.
Our theorem
\eqref{main-theorem} then tells us that  $E^{s,*}_2=0$ for $s>n$, and
furthermore provides a Koszul resolution for computing the
$E_2$-term.  In particular, makes possible explicit calculations at
height 2, some of which are described in the preprint
\cite{rezk-power-ops-structure-calculation}.

\subsection{Sketch of the proof}

We briefly indicate here the structure of the proof, which occupies
the rest of the paper, and is completed in
\eqref{subsec:proof-of-main-thm}. 
We fix a Morava $E$-theory spectrum $E$ for some height $n\geq1$.  The
argument will not apply 
directly to the ring $\Gamma$ of additive power operations for $E$,
but rather to the ring $\Delta$ mentioned earlier, which is isomorphic
to $\Gamma$ as a graded ring by \eqref{cor:gamma-delta-iso-free}.

Consider the functor $\wt{C}$ \eqref{exam:free-einfty-spaces} which
associates to a space $X$ the free non-unital
$E_\infty$-algebra on $X$:
\[
\wt{C}(X) \approx \coprod_{m\geq1} X^m_{h\Sigma_m}.
\]
There is an analogous functor $\wt{D}$ \eqref{exam:free-einfty-spectra} on
spectra, given by  
\[
\wt{D}(Y) \approx \bigvee_{m\geq1} Y^{\sm m}_{h\Sigma_m},
\]
and we have that $\Sip \wt{C}\approx \wt{D}\Sip$.

In general, given a functor $F$ from an additive category to an
abelian category, we can define a linearization
\eqref{subsec:linearization-def} of $F$, by  
\[
\lin_F (X) \defeq \Cok\left[ F(p_1+p_2)-F(p_1)-F(p_2)\colon F(X\oplus
  X)  \ra F(X)\right],
\]
where $p_1,p_2\colon X\oplus X\ra X$ are the two projections.

Let $\cE{*}X$ denote the completed $E$-homology
(\eqref{subsec:completed-e-homology} of a space $X$.
Applying linearization to the composite functor $F=\cE{*}D$, with 
$X=
S^0$, we are lead to consider the cokernel of a map
\[
\cE{*}D(S^0\vee S^0) \ra \cE{*}D(S^0)
\]
The first step \eqref{prop:linearization-of-algapprox} of the proof is
to identify this cokernel (the linearization of $F$ at $S^0$)
with the underlying $E_*$-module
of an algebra $\Delta$ of power operations (see
\S\ref{sec:rings-of-power-ops}).   
There is a
decomposition $\Delta\approx \bigoplus_{k\geq0} \Delta[k]$, where
$\Delta[k]$ comes from the $m=p^k$ summand in $C$.   (We actually
state and prove this step in terms of an algebraic approximation functor
$\talgapprox$ (\S\ref{subsec:algebraic-functor}), which has the property
that $\talgapprox(\bigoplus
E_*) \approx \cE{*}\wt{D}(\bigvee S^0)$.)

Because $\wt{D}$ is a monad on the homotopy category of spectra, we may
consider the two-sided bar construction
$\B(\wt{D})=\B(\wt{D},\wt{D},\wt{D})$.  A similar 
argument \eqref{prop:linearization-of-t-bar-is-delta-bar} identifies
the cokernel of the analogous map 
\[
\cE{*}\B(\wt{D})(S^0\vee S^0) \ra \cE{*}\B(\wt{D})(S^0)
\]
(i.e., the linearization of $F=\cE{*}\B(\wt{D})$ at $S^0$)
with the two-sided bar construction
$\B(\Delta)=\B(\Delta,\Delta,\Delta)$ of the ring $\Delta$. 
(Again, our actual statement is given in terms of the bar construction
$\B(\talgapprox)$ of the algebraic approximation functor $\talgapprox$.)

What we are actually interested in is a certain quotient $\rB(\Delta)$
of $\B(\Delta)$, which is isomorphic to $\B(E_*,\Delta,E_*)$.  This
quotient admits \eqref{subsec:bar-construction-rings} a decomposition
$\bigoplus_{k\geq0} \rB(\Delta)[k]$ 
associated to the decomposition of $\Delta$.  We observe
\eqref{subsec:koszul-def} that $\Delta$ is 
Koszul if and only if the homology of $\rB(\Delta)[k]$ is concentrated
in degree $k$, for all $k\geq0$.    More precisely,  we take
this homological vanishing property  as the
\emph{definition} of being Koszul; the discussion of
\S\ref{sec:koszul-rings} explains why this is the correct definition.
In particular, we show that with this definition, if a ring is \emph{Koszul}
then it is necessarily \emph{quadratic}
\eqref{prop:koszul-implies-quadratic}.  This also means that we do not
need to first construct an admissible basis (or any basis at all) for
$\Delta$, as is typical in many proofs of the Koszul property.

To prove that $H_*\rB(\Delta)[k]$ is concentrated in degree $k$, we
look at the combinatorics of the bar construction 
$\B(\wt{C})=\B(\wt{C},\wt{C},\wt{C})$, which are governed by
partitions.  In particular, 
there is a weak equivalence of simplicial spaces
\[
\B(\wt{C})(X) \approx \coprod_{m\geq0} (P_m\times X^m)_{h\Sigma_m}
\]
where $P_m$ is  the nerve of the poset of
partitions of an $m$-element set \eqref{subsec:partition-complex}.
(The simplicial coordinate comes 
from the simplicial set $P_m$.)  Translating this into a statement
about $\rB(\Delta)$, we discover
\eqref{cor:fundamental-observation-cor} that $\rB(\Delta)[k]$ is
isomorphic 
to the cokernel of a certain map
\[
\cE{*}(\overline{P}_m\sm (S^0\vee S^0)^{\sm m})_{h\Sigma_m} \ra
\cE{*}(\overline{P}_m \sm (S^0)^{\sm m})_{h\Sigma_m}
\]
where $m=p^k$, and $\overline{P}_m$ is certain quotient of $P_m$.
This cokernel is denoted $\widetilde{Q}(\overline{P}_m)$ in the text,
where $\widetilde{Q}(Y)$ is called the  \emph{transitive $E$-homology}
\eqref{subsec:transitive-e-homology} of a  
$\Sigma_m$-space $Y$.  That is, $\widetilde{Q}(Y)$ is the 
linearization of the functor
$X\mapsto \cE{*}(\Sip Y\sm X^{\sm m})_{h\Sigma_m}$ evaluated at
$X=S^0$.

Thus, the proof is reduced to showing that the simplicial abelian
group $\widetilde{Q}(\overline{P}_m)$ has its homology concentrated in
degree $k$.  We observe \eqref{prop:iso-to-bredon}  that the homology
of this 
simplicial abelian group is precisely the \emph{Bredon homology} of
the reduced partition complex with coefficients in the Mackey functor
defined by $Q$:
\[
H_*(\wt{Q}(\ol{P}_m)) \approx H^{\Sigma_m}_*(\ol{P}_m;Q). 
\]
The result follows via an application
\eqref{prop:adl-special-case} of a theorem of
Arone, Dwyer, and Lesh \cite{arone-dwyer-lesh-bredon-partition}.

\subsection{Historical remarks}

The classical example of an algebra of power operations which is
Koszul  is the May-Dyer-Lashof algebra of power
operations in the homology of a differential graded
$E_\infty$-algebra over $\F_p$.
That this algebra is Koszul (with respect to an appropriately chosen grading)
appears to be 
well-known, though I don't know an explicit reference; it is
implicitly proved in \cite{arone-mahowald-identity-functor}*{\S3.1}. 
The proof for the prime $2$ is an application of the PBW basis theorem
of Priddy; an adjustment needs to be made to give a proof at odd
primes.
Kuhn has an elegant unpublished 
proof that this algebra is Koszul (at least at the prime $2$) which
bypasses the need to find an admissible basis.  

Power operations were first constructed for Morava $E$-theory by Ando
\cite{ando-isogenies-and-power-ops},   
using power operations for $MU$ and suitable choices of orientations.
Soon after, Hopkins and Miller were able to show that Morava
$E$-theories are $E_\infty$-ring spectra, although it took a long time
for the technical details to be worked out.  Further work by Ando,
Hopkins, and Strickland allows for a description of the ring $\Gamma$ of power
operations, in terms of the relevant part of the $E$-cohomology of
symmetric groups.  The key result here is Strickland's identification
of a quotient of $E^*B\Sigma_{p^r}$ as the ring classifying subgroups
of a formal group 
\cite{strickland-finite-subgroups-of-formal-groups},
\cite{strickland-morava-e-theory-of-symmetric}.  Some exposition of
these results is given in \cite{rezk-congruence-condition}.  

The ring of power operations for height $1$ Morava $E$-theories
amounts to the case of $p$-adic $K$-theory; this case is understood by
work of McClure \cite{mcclure-dyer-lashof-k-theory}.
The first (partial) calculation of a power operations algebra for
height $2$ was carried out by Kashiwabara
\cite{kashiwabara-k2-homology}.  What he really did is find 
a basis for the Morava $K(2)$-homology of symmetric groups.  In our
language he did this by understanding $\Gamma/(p,v_1)\Gamma$.  The ideal
$(p,v_1)\Gamma\subset \Gamma$ is not a two-sided ideal, so
$\Gamma/(p,v_1)\Gamma$ is 
not actually a ring.  Thus Kashiwabara (aware of this) only computed
a product up an indeterminacy.  His calculations nonetheless make clear that at
height $2$, the ring $\Gamma$ is a quadratic algebra, and that
$\Gamma$ has an ``admissible basis'' in terms of certain monomials in
the generators, and that the algebra should satisfy the PBW condition
of Priddy; thus $\Gamma$ is Koszul at height $2$.

Ando, Hopkins, and Strickland conjectured that there is a small
resolution (of length $n$) for modules over the ring of power
operations of a height $n$ Morava $E$-theory, for any $n$; that is, that these
rings are Koszul.  
Moreover, they
explicitly describe a model for this resolution; the description
involves the ``building complex'' of the finite subgroup schemes of a
formal group.   A brief discussion of these ideas are given in
\cite{strickland-finite-subgroups-of-formal-groups}*{\S14}.
We do not address their ``building complex''
construction in this paper; however, in \cite{rezk-modular-complexes}
we have described a version of the building complex for height $n=2$,
using elliptic curves.

I announced the theorem of this paper in a talk in Mainz in 2005.
I later realized that the proof I believed I had at that time was not
complete; I found a 
correct proof in 2008, which was posted to the arXiv in 2012.  In 2012 Kathryn Lesh pointed out to me that
her work with Arone and Dwyer \cite{arone-dwyer-lesh-bredon-partition}
led to a significant simplification of the argument, which is
incorporated in this version.

\subsection{Acknowledgments}

The ideas in the proof given here owe a great deal to the work of
Arone and Mahowald on the Goodwillie tower of the identity functor
\cite{arone-mahowald-identity-functor}.  I am of course indebted to
the work of Ando, Hopkins, and Strickland on power operations for
Morava $E$-theory, which is the foundation of the present work.
Finally, I would like to thank Greg Arone, Bill Dwyer, and Kathryn
Lesh for sharing their work on the Bredon homology of the partition
complex.  Finally, I'd like the many people who over the years have
personally educated me about the algebraic theory of power operations,
a list which includes (but is not limited to): Mike Hopkins, Neil
Strickland, Matt Ando, and Paul Goerss.

The author was supported under NSF grants DMS-0203936, DMS-0505056,
and DMS-1006054 during the course of this project.

\section{Monads and bar constructions}

In this section we set up some properties of exponential monads, of
which the key example is the algebraic approximation functor
$\algapprox$ which governs power operations of Morava $E$-theory, to
be discussed in \S\ref{sec:rings-of-power-ops}.  In particular, we
describe how a suitable grading of an exponential monad determines a
grading of its bar-complex.

\subsection{Exponential monads}
\label{subsec:exponential-monads}

Let $\mathcal{C}$ be a symmetric monoidal category with monoidal
product $\otimes$ and unit $\bfone$, and suppose also that
$\mathcal{C}$ admits finite coproducts (denoted ``$\oplus$'', with
initial object $0$), and that $\otimes$ distributes over coproducts.
For convenience, we also assume that inclusions of direct summands are
always monomorphisms in $\mathcal{C}$.
By an \dfn{exponential monad}, we mean a monad equipped with natural
isomorphisms 
\[
\upsilon \colon \bfone \ra T(0), \qquad \zeta\colon T(X)\otimes T(Y)
\ra T(X\oplus Y),
\]
where the map $\zeta$ is a natural transformation of functors
$\mathcal{C}\times \mathcal{C}\ra \mathcal{C}$, with the property that 
$(\upsilon, \zeta)$ make $T\colon \mathcal{C}^\oplus \ra
\mathcal{C}^{\otimes}$ into a strong symmetric monoidal functor.
Furthermore, we require that every $T$-algebra $(A,\phi\colon TA\ra A)$ is
naturally a commutative 
monoid object in the symmetric monoidal category $\mathcal{C}$, with
unit $\bfone\xra{\upsilon} T(0)\xra{T(0)} T(A)\xra{\psi} A$ and product
$A\otimes A \xra{\eta\otimes \eta} TA\otimes TA \xra{\zeta} T(A\oplus
A) \xra{T(\nabla)} TA \xra{\psi} A$. 

The canonical example of an exponential monad is the free commutative
algebra monad on the category of abelian groups.  The examples we need
to work with will be free $E_\infty$-algebra monads on some
homotopy category of spaces or spectra, or monads derived from such.

\subsection{Graded exponential monads}
\label{subsec:graded-exponential}

By a \dfn{graded exponential monad}, we mean an exponential monoidal
monad $T$, together with functors $T\langle m\rangle\colon
\mathcal{C}\ra \mathcal{C}$ and natural monomorphisms $\gamma_m\colon
T\langle m\rangle(X) \ra T(X)$, which fit together to give a direct
sum decomposition  
\[
(\gamma_m)\colon \bigoplus_{m\geq0} T\langle m\rangle (X) \xra{\sim} T(X)
\]
and such that there exist (necessarily unique, because the $\gamma_k$
are monomorphisms) dotted arrows in
\[\xymatrix{
{\bfone} \ar@{.>}[r]^-{\upsilon_0} \ar[dr]_{\upsilon} & {T\langle
  0\rangle(0)} \ar[d]^{\gamma_0} 
& {T\langle p\rangle(X) \otimes T\langle q\rangle(Y)} \ar@{.>}[r]^{\zeta_{p,q}}
\ar[d]_{\gamma_p\otimes \gamma_q} 
& {T\langle p+q \rangle (X\oplus Y)} \ar[d]^{\gamma_{p+q}}
\\
& {T(0)}
& {T(X)\otimes T(Y)} \ar[r]_{\zeta} & {T(X\oplus Y)}
}\]
\[\xymatrix{
{X} \ar@{.>}[r]^-{\eta_1} \ar[dr]_{\eta} & {T\langle 1\rangle (X)}
\ar[d]^{\gamma_1}
& {T\langle p\rangle T\langle q\rangle (X)} \ar@{.>}[r]^-{\mu_{p,q}}
\ar[d]_{\gamma_p \gamma_q} & {T\langle pq\rangle (X)}
\ar[d]^{\gamma_{pq}}
\\
& {T(X)}
& {TT(X)} \ar[r]_{\mu} & {T(X)}
}\]
such that $T\langle 0\rangle(0) \ra T\langle 0
\rangle (X)$ and $\eta_1\colon X\ra T\langle 1\rangle(X)$ are
isomorphisms for all objects $X$.
These conditions imply that 
\[
\bfone \ra T\langle 0\rangle (X),\qquad
(\zeta_{p,q})\colon \bigoplus_{p+q=m}
T\langle p\rangle(X) \otimes T\langle q \rangle (Y)
\ra T\langle p+q\rangle (X\oplus Y)
\]
are isomorphisms.
Furthermore, each composite $T^{\circ q}$ admits a direct sum
decomposition $T^{\circ q}\approx \bigoplus_{m\geq0} T^{\circ
  q}\langle m\rangle$, determined inductively by
\[
T^{\circ q}\langle m\rangle (X) \approx 
\bigoplus_{m=\sum_j jm_j}
\left[
\bigotimes_{j\geq0} T\langle m_j \rangle (T^{\circ (q-1)}\langle
j\rangle (X)) \right],
\]
and this decomposition is compatible with the structure maps of the
monad.  That is, each map $T^{\circ i}\circ \mu\circ T^{\circ j} \colon
T^{\circ (i+j+1)} \ra T^{\circ (i+j)}$ restricts to a coproduct of
maps $T^{\circ (i+j+1)}\langle m\rangle \ra T^{\circ (i+j)}\langle
m\rangle$, and similarly for $T^{\circ i}\circ \eta\circ T^{\circ
  j}$. 

We refer to $T^{\circ q}\langle m\rangle$ as the \dfn{weight
  $m$} part of $T^{\circ q}$.  Note that if $m=\prod_{i=1}^q m_i$,
then there is an evident map 
\[
T\langle m_1\rangle\circ \cdots \circ  T\langle m_q\rangle \ra T^{\circ q}\langle
m\rangle,
\]
which is an inclusion of a direct summand.  We say that summands of
this form have \dfn{pure weight $m$}.

\subsection{The positive part of a graded exponential monad}
\label{subsec:positive-part}

Given a graded exponential monad $T$ on $\mathcal{C}$, we write
$\wt{T}$ for the subfunctor $\wt{T}(X) = \bigoplus_{m\geq1} T\langle
m\rangle (X)$ consisting of the part in positive weight.  It is clear
that $\wt{T}$ is a monad in its own right, so that the inclusion map
$\wt{T}\ra T$ is a map of monads.  We can similarly speak of the
weight $m$ part $\wt{T}^{\circ q}\langle m\rangle$ of $\wt{T}^{\circ
  q}$,  which will be a subobject of $T^{\circ q}\langle m\rangle$.
(Note that $\wt{T}^{\circ q}\langle m\rangle$ will not generally be
equal to $T^{\circ q}\langle m\rangle$ if $q\geq2$, as the latter
contains a large number of contributions from the weight $0$ part of
$T$.)

\subsection{Bar complexes}
\label{subsec:bar-complex}

Given a monad $(T,\eta\colon I\ra T,\mu\colon T\circ T\ra T)$ on a
category $\mathcal{C}$, we will 
write $\B(T)=\B(T,T,T)$ for the two-sided bar construction for $T$;
this is an augmented simplicial object in endofunctors of
$\mathcal{C}$, with $\B_q(T) = T^{\circ (q+2)}$.  More generally, given
$M,N\colon 
\mathcal{C}\ra \mathcal{C}$ which are left and right 
modules for $T$ respectively, there is a bar construction $\B(M,T,N)$
with $\B_q(M,T,N)\approx M\circ T^{\circ q}\circ N$.

Now suppose $T$ is a graded exponential monad.  For each $q\geq0$ and
$m\geq 0$ we define
\[
\B_q(T)\langle m\rangle \defeq T^{\circ(q+2)}\langle m\rangle,
\]
using the direct sum decomposition $T^{\circ (q+2)} \approx
\bigoplus_{m\geq0} T^{\circ (q+2)}\langle m\rangle$ described above.
Thus, the simplicial endofunctor $\B(T)\approx \bigoplus_{m\geq0}
\B(T)\langle m\rangle$ admits a weight decomposition.

We may similarly consider the bar construction of the positive part
$\wt{T}$, and we similarly obtain a weight decomposition $\B(\wt{T})\approx
\bigoplus_{m\geq1} \B(\wt{T})$.

\begin{exam}\label{exam:free-einfty-spaces}
Let $O$ be an $E_\infty$-operad in spaces, and let $C$ denote the 
monad on spaces defined by $C(X) \defeq \coprod_{m\geq0} (O(m)\times
X^m)_{h\Sigma_m}$.  The functor $C$ descends to a monad on the
homotopy category $h\Spaces$ of spaces, which we also denote by
$C$.  This is a graded exponential monad in
our sense; the graded pieces are $C\langle m\rangle (X) \approx
(O(m)\times X^m)_{h\Sigma_m}$.

The corresponding bar complex admits a grading $\B(C)\approx
\coprod_{m\geq0} \B(C)\langle m\rangle$; applied to a space, we obtain
a decomposition $\B(C)(X)\approx \coprod_{m\geq0} \B(C)\langle
m\rangle (X)$ of simplicial spaces.  As we will note below
\eqref{prop:iso-bar-c-with-partition-general}, for the positive part
of $C$ there is a  natural
weak equivalence 
\[
\B(\wt{C})\langle m\rangle (X) \approx (P_m\times X^m)_{h\Sigma_m}
\]
of simplicial spaces, where $P_m$ is the partition complex
\eqref{subsec:partition-complex} on the set
of $m$ elements.
\end{exam}

\begin{exam}\label{exam:free-einfty-spectra}
The monad $D$ on the homotopy category of spectra, defined by
$D(Y)\defeq \bigvee_{m\geq0} (O(m)_+\sm Y^{\sm m})_{h\Sigma_m}$, is
similarly an exponential monad.  
\end{exam}

\begin{exam}\label{exam:free-e-algebra}
Let $\Pfree$ denote the free commutative $E$-algebra monad, defined on
the category $\Mod{E}$ of $E$-module spectra.  This functor descends
to a functor on the homotopy category $h\Mod{E}$, which we also denote
by $\Pfree$.  As such, it is a graded exponential functor, with
$\Pfree\approx \bigvee_{m\geq 0}\Pfree\langle m\rangle$, where 
$\Pfree\langle m\rangle(M) \approx (M^{\sm_E m})_{h\Sigma_m}$.  We
will typically write $\Pfree_m$ for $\Pfree\langle m\rangle$.  
\end{exam}

\subsection{Associations}
\label{subsec:associations}

Let $T$ and $T'$ be graded exponential monads on suitable categories
$\mathcal{C}$ and $\mathcal{C}'$.  An \dfn{association} from $T$ to
$T'$ is a functor $G\colon\mathcal{C}\ra \mathcal{C}'$ which is
equipped with the structure of a weak monoidal functor in two
different ways, namely as functors $C_\oplus \ra C_\oplus'$ and
$C_\otimes\ra C_\otimes'$, together with a natural map $TG\ra 
GT'$ which is compatible with all the structure.

\begin{exam}
The functor $\Sip\colon h\Spaces \ra
h\Spectra$ defines as association between $C$ and $D$.
Likewise, the functor $\Sigma\colon h\Spectra\ra h\Spectra$ defines an
association between $D$ and itself.
\end{exam}

\begin{exam}
There is an association between the monads $C$ and $\Pfree$ described
above, given by the functor $E\sm \Sip\colon h\Spaces\ra h\Mod{E}$.  
In particular: (i) $E\sm \Sip$ takes coproducts to
coproducts; (ii) $E\sm \Sip$ takes products to smash products; (iii)
there is a natural map (in fact, a weak equivalence)
\[
\Pfree(E\sm \Sip X) \ra E\sm\Sip C(X),
\]
which is compatible with the exponential structures on the monads, and
which is compatible with the gradings, in the sense that it restricts
to maps  $E\sm \Sip C\langle m\rangle (X) \ra
\Pfree\langle m\rangle (X)$.
\end{exam}

\begin{exam}
One more example of exponential monad is given in the next section,
where we describe a monad $\algapprox$ on the category $\Mod{E_*}$ of
$E_*$-modules.  There is an association between $\algapprox$ and the
monad $\Pfree$ above,  given by $\pi_*L_{K(n)}\colon h\Mod{E}\ra \Mod{E_*}$, so
that  there is a natural map
\[
\algapprox(\pi_*L_{K(n)}M)\ra \pi_*L_{K(n)}\Pfree(M). 
\]
\end{exam}

\section{Rings of power operations}
\label{sec:rings-of-power-ops}

The homotopy groups of a $K(n)$-local commutative $E$-algebra spectrum
are naturally algebras over a certain monad $\algapprox$, which
captures algebraically the $E$-homology of symmetric groups.
In this section, we recall from \cite{rezk-congruence-condition}
properties of the monad $\algapprox$. 
From this monad, we will extract two kinds of graded rings of ``power
operations''; the ring $\Gamma^q$, which is a ring of
\emph{additive} operations on $\pi_{-q}$ of a $K(n)$-local commutative
$E$-algebra, and the ring $\Delta^q$, which is a ring of operations
on the degree $-q$ part of the  cotangent space of an augmented
algebra.  
The main result of 
this section is that all of the rings in question are isomorphic.  It
is the ring $\Delta\defeq \Delta^0$ which we will 
explicitly show is Koszul in subsequent sections.

\subsection{Introduction}

Before diving in to the case of Morava $E$-theory, it may be useful to
consider things in a more  general context.  

Fix an arbitrary structured commutative ring spectrum $E$, and
the category $\Alg{E}$ of commutative $E$-algebra spectra.  Consider
the problem of determining all \emph{natural 
  operations} on the homotopy groups of commutative $E$-algebras.
That is, for all $i,j\in \Z$ we would like to compute the set of
natural transformations
\[
\pi_i(-) \ra \pi_j(-) \quad \text{of functors $h\Alg{E}\ra \Set$.}
\]
We might instead restrict to the case of \emph{additive} operations,
i.e., natural transformations of functors $h\Alg{E}\ra \Ab$. 

Let us consider how this works out  in the case that $E=H\F_p$, the
mod $p$ Eilenberg MacLane spectrum.\footnote{This is inspired by the
treatments in \cite{may-general-algebraic-approach} and  in
\cite{bmms-h-infinity-ring-spectra} (especially \S~ 
IX.2), though not identical to what those authors write, since after
all the foundations of 
structured module spectra were not available at that time.}  
\begin{itemize}
\item Commutative $H\F_p$-algebras are algebras for the free symmetric
  $H\F_p$-algebra monad $\Pfree$ on $\Mod{H\F_p}$, which as described above
  \eqref{exam:free-e-algebra} is an example of an exponential monad.

\item
  This monad  descends to an exponential monad on the   homotopy category
  $h\Mod{H\F_p}$ of modules, which we also denote by $\Pfree$.  The
  underlying object in $h\Mod{H\F_p}$ of a 
  commutative $H\F_p$-algebra spectrum is naturally an algebra for
  $\Pfree$ on $h\Mod{H\F_p}$ (this kind of structure is called an 
  $H_\infty$-$H\F_p$-algebra structure).

\item Taking homotopy groups defines  a symmetric monoidal
  equivalence $\pi_*\colon   h\Mod{H\F_p}\xra{\sim}  \Mod{\F_p}^*$ to
  the category of graded $\F_p$-vector spaces.  

\item Thus the monad $\Pfree$ on $h\Mod{H\F_p}$ descends to a monad
  $\algapprox$ on $\Mod{\F_P}^*$, which we regard as an ``algebraic
  approximation'' to $\Pfree$.  It is a very good approximation
  indeed, as by construction it comes with  a monoidal natural isomorphism
\[
\alpha \colon \algapprox(\pi_* M) \xra{\sim}  \pi_*\Pfree(M)
\]
of functors $h\Mod{H\F_P}\ra \Mod{\F_P}^*$, which is an example of an
association between the monads $\Pfree$ and $\algapprox$ as defined in
\eqref{subsec:associations}.  Furthermore, $\algapprox$ inherits from
$\Pfree$ the structure of an exponential monad.

\item The object $\algapprox(\F_p[i])\approx \pi_*\Pfree(\Sigma^i
  H\F_p)$ is the free $\algapprox$-algebra on one generator in degree
  $i$.   (Here we write $V[i]$ for the shifted graded vector space,
  with $V[i]_j=V_{-i+j}$.)  It is a formal consequence that the set of natural
  transformations $\pi_i(-)\ra \pi_j(-)$ of functors $h\Alg{H\F_p}\ra
  \Set$ is equal to 
\[
\pi_j \Pfree(\Sigma^iH\F_p) \approx \bigl( \algapprox(\F_p[i])\bigr)_j,
\]
the degree $j$ part of the free algebra on a generator in degree $i$.

\item Addive operations, i.e., natural transformations $\pi_i(-)\ra
  \pi_j(-)$ of functors $h\Alg{H\F_p}\ra \Ab$, correspond precisely to the
  \emph{primitive} elements in the free algebra, i.e., elements in the
  equalizer of the pair of maps 
\[
\algapprox(i_1+i_2), \algapprox(i_1)+\algapprox(i_2)\colon
\algapprox(\F_p[i])_j \rightrightarrows \algapprox(\F_p[i]\oplus \F_p[i])_j,
\]
where $i_1,i_2\colon \F_p[i]\ra \F_p[i]\oplus \F_p[i]$ are the
evident  inclusion maps.
\end{itemize}
Thus, the theory of homotopy operations for commutative
$H\F_p$-algebra spectra is entirely described by  an essentially
algebraic object, the monad $\algapprox$ on graded vector spaces.  The
structure of $\algapprox$ and its category of algebras is entirely
understood: it is the category of graded $\F_p$-algebras equipped with
\emph{Dyer-Lashof operations} $\{Q^s, \beta Q^s\}$, which satisfy an
explicit set of 
axioms\footnote{Essentially
  \cite{bmms-h-infinity-ring-spectra}*{Ch.\ III, Thm.\ 1.1, (1)--(7)}, with the
  modification that in the homotopy of commutative $H\F_p$-algebras
  there is no Bockstein operation $\beta$; the Bockstein is defined
  only on $H\F_p$-algebras of the form $H\F_p\sm_S A$ for commutative
  $S$-algebras $A$.  Rather, one posits additional 
  operations ``$\beta Q^s$'' which satisfy suitable identities.  See
  \cite{may-general-algebraic-approach} for a concrete description of
  this structure.  }.

\begin{rem}
This story goes much the same way for algebras over the rational Eilenberg-MacLane
spectrum $H\Q$.  In this case, the monad $\algapprox$ on the category
of graded rational vector spaces turns out to be identical to the
usual (graded) symmetric algebra monad.  
\end{rem}

We can  carry out a variant of the scenario in which $H\F_p$ is replaced
by a Morava $E$-theory spectrum.  This was done in
\cite{rezk-congruence-condition}, and we will summarize the main
points we need below.  Before doing so, we note some of the points in
the above outline
which require modification. 
\begin{itemize}
\item The homotopy category $h\Mod{E}$ of $E$-modules is not
  equivalent to the category of graded $E_*=\pi_*E$-modules.
  However, there is a full subcategory $h\Mod{E}^{\mr{free}}\subset
  h\Mod{E}$ of 
  ``free'' modules (i.e., those for which $\pi_*M$ is a free
  $E_*$-module),  which is equivalent to the full subcategory
  $\Mod{E_*}^{\mr{free}}\subset \Mod{E_*}$ of free graded
  $E_*$-modules.

\item For a free $E$-module $M\approx \bigvee \Sigma^{i_\alpha}E$, it
  is rarely the case that $\Pfree(M)\approx \bigvee_m M^{\sm_E
    m}_{h\Sigma_m}$ is free.  This is not specific to 
  Morava $E$-theory:  for instance,  $E^{\sm_E m}_{h\Sigma_m}
  \approx E\sm \Sip B\Sigma_m \approx E\vee (E\sm \Si B\Sigma_m)$, and
  $\pi_*(E\sm \Si B\Sigma_m)$  is torsion for any homology theory
  $E$.

\item However, it is true \cite{rezk-congruence-condition}*{Prop.\
    3.17} that for a \emph{finitely generated} free
  $E$-module $M$, the spectra $L_{K(n)} \Pfree_m(M)\approx
  L_{K(n)}( M^{\sm_E m}_{h\Sigma_m}) $ obtained by taking the
  $K(n)$-localization of the weight-homogenous pieces of $\Pfree(M)$
  are also finitely generated and free.  

\item The $K(n)$-localization functor $L_{K(n)}$ does not preverve
  coproducts of spectra.  Thus, $L_{K(n)}\Pfree(M)$ is not generally a
  free $E$-module if $M$ is free.  What is true in this case is that
  $\pi_*L_{K(n)}   \Pfree(M)$ is the $\mathfrak{m}$-adic completion of
  a free $E_*$-module, where $\mathfrak\subset E_0$ is the maximal
  ideal of the coefficient ring.

\end{itemize}

For the purposes of getting at rings of additive operations on
$K(n)$-local commutative $E$-algebras, it
suffices to deal with $L_{K(n)}\Pfree_m(M)$ for finite free
$E$-modules $M$.  Thus we will define ``algebraic approximation'' functors
$\algapprox_m \colon \Mod{E_*}\ra \Mod{E_*}$ by a Kan extension of
$\pi_*L_{K(n)}\Pfree_m$ restricted to finite free $E$-modules.  We
will also construct a natural ``approximation map''
\[
\alpha \colon \algapprox\pi_*(-) =\bigoplus_m \algapprox_m\pi_*(-) \ra \pi_*L_{K(n)}\Pfree(-)
\]
of functors $h\Mod{E}\ra\Mod{E_*}$, which is an isomorphism for finite
free $E$-modules.  

\begin{rem}
It is important to note that $\alpha$ is \emph{not} generally an
isomorphism, even on finite free $E$-modules.  What is true is that
for a free (but not necessarily finite) $E$-module $M$, the induced map
\[
(\algapprox\pi_*M)^{\sm}_{\mathfrak{m}} \xra{\sim} \pi_*L_{K(n)}\Pfree(M)
\]
from  completion with respect to the maximal ideal
$\mathfrak{m}\subset E_0$ is an isomorphism.   
We address this issue and its relevance to operations
in \eqref{subsec:aside-on-additive-ops}.

Another approach to this issue is to construct a ``better''
approximation functor,  defined on a suitable category of
\emph{completed} $E_*$-modules.  Such an approximation functor comes
with an approximation map which is an isomorphism on all (completed)
free $E$-modules.  Such a better approximation functor is constructed
by Barthel and Frankland
\cite{barthel-frankland-completed-approximation}.   This is a useful
construction for the sake of applications,  but does not aid in
proving the theorem of this paper, so we will not make use of it.
\end{rem}

\subsection{Completed $E$-homology}
\label{subsec:completed-e-homology}

We now fix a Morava $E$-theory spectrum, with height $n$.

Given a spectrum $Y$, we define the \dfn{completed $E$-homology} of
$Y$ by 
\[
\cE{*}Y \defeq \pi_* L_{K(n)}(E\sm Y).
\]
If $X$ is a space, we write $\cE{*}X$ for $\cE{*}\Sip X$.
As is well-known, the functor $\cE{*}$ satisfies the
Eilenberg-Steenrod axioms but not Milnor's wedge axiom.

\subsection{Algebraic structure on the homotopy groups of
  $E$-algebras} 
\label{subsec:algebraic-functor}

As we noted above \eqref{exam:free-e-algebra}, the free commutative
$E$-algebra functor $\Pfree\colon h\Mod{E}\ra h\Mod{E}$ carries the
structure of a graded exponential monad with respect to derived smash
product of $E$-modules; in particular, it admits a
decomposition $\Pfree\approx \bigvee_{m\geq0} \Pfree_m$.  We need one
more piece of structure; namely, for all $m\geq1$, there is a natural
transformation $e\colon \Sigma \Pfree_m\ra \Pfree_m\Sigma$, defined
because $\Pfree_m$ is compatible with the enrichment of $\Mod{E}$ over
pointed spaces.

Following \cite{rezk-congruence-condition}*{\S4.4} we define the
\dfn{algebraic approximation functor} $\algapprox_m\colon
\Mod{E_*}\ra \Mod{E_*}$ as a left Kan extension
\[
\algapprox_m := \LKan_{\pi_* i} \pi_*L_{K(n)}\Pfree_mi
\]
with respect to the functors in the diagram
\[\xymatrix@C=40pt{
{h\Mod{E}^{\mr{finite\,free}}} \ar[r]^-{i} \ar[d]_{\pi_*}^{\sim}
& {h\Mod{E}} \ar[r]^-{\pi_*L_{K(n)}\Pfree_m} \ar[d]_{\pi_*}
& {\Mod{E_*}}
\\
{\Mod{E_*}^{\mr{finite\,free}}} \ar[r]_-{j}
& {\Mod{E_*}} \ar@{.>}[ur]_{\algapprox_m}
}\]
Here $i$ and $j$ are the evident fully faithful inclusions.
By construction (since $\pi_* i\approx j\pi_*$ is fully faithful)
there is a natural isomorphism $\algapprox_m \pi_* i \approx \pi_*
L_{K(n)}\Pfree_m i$, i.e., $\algapprox_m(\pi_*M)$ computes the homotopy
groups of $L_{K(n)}\Pfree_m(M)$ when $E$ is a finitely generated free
$E$-module.  In \cite{rezk-congruence-condition}*{\S4.6} it is shown
that this equivalence extends to an \dfn{approximation map}
\[
\alpha_m \colon \algapprox_m(\pi_*M) \ra \pi_*L_{K(n)}\Pfree_m(M),
\]
a natural transformation of  functors $h\Mod{E}\ra \Mod{E_*}$.

We define $\algapprox:= \bigoplus_m \algapprox_m\colon \Mod{E_*}\ra
\Mod{E_*}$, and obtain  an approximation map
\[
\alpha\colon \algapprox(\pi_* M) \ra \pi_* L_{K(n)}\Pfree(M)
\]
(using the natural map $\bigvee_m L_{K(n)}\Pfree_m \ra L_{K(n)}(
\bigvee_m \Pfree_m)=L_{K(n)}\Pfree$).

\subsection{Properties of the algebraic approximation}
\label{subsec:algapprox-properties}

We summarize here the salient properties of the 
of the algebraic approximation.
\begin{enumerate}
\item The functor $\algapprox$ is equipped with the structure of an
  exponential monad with respect to tensor product of $E_*$-modules
  \cite{rezk-congruence-condition}*{Prop.\ 4.10, \S4.13},
  and the approximation map $\alpha$ \cite{rezk-congruence-condition}*{\S4.6} is an association between
  $\algapprox$ and $\Pfree$ in the sense of 
  \eqref{subsec:associations}.  (This is clear from the construction of
  $\gamma_k$ in \cite{rezk-congruence-condition}*{\S4.13}.)

\item The decomposition $\algapprox = \bigoplus_m \algapprox_m$ gives
  the structure of a  \emph{graded} exponential monad,
\cite{rezk-congruence-condition}*{\S4.4,
    \S4.13}, corresponding to the analogous decomposition of
  $\Pfree$.  We will sometimes write 
  typically we write $\algapprox\langle m\rangle$ for 
  $\algapprox_m$.  Furthermore, the approximation map
  descends to the grading, via the maps $\alpha_m$.

\item 
If $M$ is an $E$-module such that $\pi_*M$ is finite and free as an
$E_*$-module, then 
$\alpha_m$ evaluated at $M$ is an isomorphism
\cite{rezk-congruence-condition}*{Prop.\ 4.8}.

\item The approximation maps $\alpha_m$ give rise to a natural
  transformation
\[
\wh\alpha\colon L_0 \algapprox(\pi_*M) \ra \pi_*L_{K(n)}\Pfree(M)
\]
which is an isomorphism when $\pi_*M$ is a flat $E_*$-module
\cite{rezk-congruence-condition}*{Prop.\ 4.17}.  Here
$L_0$ denotes the $0$th left derived functor of the
$\mathfrak{m}$-adic completion functor $(-)^{\sm}_{\mathfrak{m}}\colon
\Mod{E_*}\ra \Mod{E_*}$, where $\mathfrak{m}\subset \pi_0E$ is the
maximal ideal.  

We note that if $M_*$ is a free $E_*$-module, then
$\algapprox(M_*)$ is also free, whence $L_0\algapprox(M_*)$ is
isomorphic to the $\mathfrak{m}$-adic completion of the free module
$\algapprox(M_*)$.  Thus, if $\pi_*M$ is a free $E_*$ module, then we
get an isomorphism $\algapprox(\pi_*M)^{\sm}_{\mathfrak{m}}\xra{\sim}
\pi_*L_{K(n)}\Pfree(M)$. 

\item 
If $N$ is a finite and free $E_*$-module, then so is $\algapprox_m N$
for all $m\geq0$ (by (3) and \cite{rezk-congruence-condition}*{Prop.\
  3.16}).  Thus for all $q\geq0$, the approximation maps 
induce isomorphisms 
\[
\algapprox^{\circ q}\langle m\rangle(\pi_*M)\xra{\sim}
\pi_*L_{K(n)}(\Pfree^{\circ q}\langle m\rangle M)
\]
when $\pi_*M$ is finite and free.

\item 
The functor $\algapprox$ and each functor $\algapprox_m$ commutes with
filtered colimits and reflexive coequalizers
\cite{rezk-congruence-condition}*{Prop.\ 4.12}.  

\item
If $M$ is an $E_*$-module concentrated in even degree, then so is
$\algapprox M$.  (This follows using the Kan extension property (5),
the isomorphism (3), and the fact that $\cE{*}B\Sigma_m$ are
concentrated in even degree.)
\end{enumerate}

We write $\gralgcat$ for the category of algebras for the monad
$\algapprox$ on $\Mod{E_*}$.  

\begin{enumerate}
\item [(8)] The category $\gralgcat$ is complete and
  cocomplete; limits are computed in the underlying category
  $\Mod{E_*}$, and colimits are computed in the underlying category
  $\Alg{E_*}$ of commutative $E_*$-algebras.  In particular,
  coproducts in $\gralgcat$ are tensor products of $E_*$-modules
  \cite{rezk-congruence-condition}*{Cor.\ 4.19}.  

\item [(9)] Using the approximation map, we see
  that every algebra $A$ for the monad $\Pfree$ on $h\Mod{E}$ (i.e., for
  every $H_\infty$-$E$-algebra), the homotopy groups $\pi_*L_{K(n)}A$
  naturally carry the structure of a $\algapprox$-algebra.  That is,
  we obtain a functor
\[
\pi_*L_{K(n)} \colon \Alg{\Pfree} \ra \gralgcat
\]
lifting $\pi_*L_{K(n)} \colon \Mod{E}\ra \Mod{E_*}$. 

\item [(10)]
Let $\Sigma\colon \Mod{E_*}\ra \Mod{E_*}$ denote the functor $\Sigma
M\defeq E_*S^1\otimes_{E_*} M$.  For $m\geq1$, there is an algebraic
\dfn{suspension   map}
\[
E\colon \Sigma \algapprox_m \ra \algapprox_m \Sigma,
\]
which with respect to the approximation map is compatible with the
topological suspension map $e\colon 
\Sigma \Pfree_m\ra \Pfree_m \Sigma$.  Furthermore, the algebraic
suspension map is compatible with the monad structure on $\algapprox$,
in the sense that both ways of building a map $\Sigma
(\algapprox\algapprox)\langle m\rangle \ra (\algapprox\langle m\rangle)\Sigma$
coincide.  
(This suspension map is denoted
$E_1$ in \cite{rezk-congruence-condition}*{\S4.25}, where the notation
$\omega^{-q/2}$ is used for $E_*S^q$.  Compatibility between algebraic
and topological suspension maps is by construction. The algebraic
suspension map is defined for finite free $E_*$-modules by means of the
approximation isomorphism and the topological suspension map; the
general algebraic suspension map is then defined using (5) above.  The
compatibility with the monad structure is
\cite{rezk-congruence-condition}*{Prop.\ 4.27}.) 
\end{enumerate}

We need one more fact about the algebraic suspension: it ``kills
decomposable elements''.
\begin{prop}\label{prop:diagonal-is-null}
For $i,j\geq1$, the composition
\[
\Sigma(\algapprox_i M\otimes \algapprox_j N) \xra{\Sigma\zeta} \Sigma
\algapprox_{i+j}(M\oplus N) \xra{E} \algapprox_{i+j}
\Sigma(M\oplus N)
\]
is equal to $0$.
\end{prop}
\begin{proof}
Consider functors $\Mod{E}\times \Mod{E}\ra \Mod{E}$ defined by $(X,Y)\mapsto
\Pfree_iX\sm_E \Pfree_jY$ and $(X,Y)\mapsto \Pfree_{i+j}(X\vee Y)$.  
These are each enriched over pointed spaces, and thus we obtain a
homotopy commutative diagram
\[\xymatrix{
{\Sigma (\Pfree_i X \sm \Pfree_j Y)} \ar[r] \ar[d]_{\Sigma\zeta}
& {\Pfree_i \Sigma X \sm \Pfree_j \Sigma Y} \ar[d]
\\
{\Sigma \Pfree_{i+j}(X\vee Y)} \ar[r]_-{e}
& {\Pfree_{i+j}(\Sigma X\vee \Sigma Y)}
}\]
where the vertical maps come from the exponential structure.    We see
that the bottom horizontal map is the suspension map $e$, while the
top horizontal map factors $\Sigma(\Pfree_i X\sm \Pfree_j Y) \ra
\Sigma\Pfree_i X\sm \Sigma \Pfree_j Y\ra \Pfree_i \Sigma X\sm
\Pfree_j\Sigma Y$.  The first of these two maps is null, as it uses
the diagonal embedding $S^1\ra S^1\sm S^1$.  Thus the composite
$e\circ \zeta$ is null.  The corresponding vanishing result for the
algebraic suspension map is immediate for finite free modules, and
follows in general since the algebraic approximation functors are left
Kan extended from finite free modules.
\end{proof}

\subsection{Augmented rings}
\label{subsec:augmented-rings}

Let $\gralgcat_{/E_*}$ denote the category of $\algapprox$-objects
augmented over $E_*$; an object of $\gralgcat_{/E_*}$ is a morphism $A\ra
E_*$.

We will write $\talgapprox$ for the subfunctor of $\algapprox$
defined by 
\[
\talgapprox(M) \defeq \bigoplus_{m\geq1} \algapprox_m(M). 
\]
This is precisely the kernel of the natural augmentation map
$\algapprox(M)\ra \algapprox(0)=E_*$.

As noted above (\eqref{subsec:positive-part}, using
\eqref{subsec:algapprox-properties}(2)),  
the functor $\talgapprox$ itself inherits the structure of a monad on
$\Mod{E_*}$.  The  category of algebras over $\talgapprox$ is equivalent to the
category of 
$\gralgcat_{/E_*}$ of augmented $\algapprox$-algebras.  This is
a standard 
observation, so we won't spell out the details, except to note
that if $A\ra E_*$ is an augmented $\algapprox$-algebra with structure
map $\psi\colon \algapprox A\ra A$ and augmentation ideal $\wt{A}$, then
the corresponding $\talgapprox$-algebra structure $\tilde\psi\colon
\talgapprox \wt{A}\ra \wt{A}$ is simply the restriction of $\psi$ to
$\talgapprox \wt{A}\subset \algapprox A$.

\subsection{Abelian group objects}
\label{subsec:abelian-group-objects}

Recall that the notion of \dfn{abelian group object} can be defined in
any category with finite products.

\begin{prop}
An object $A\ra E_*$ of $\gralgcat_{/E_*}$ with augmentation ideal $\wt{A}$
admits the structure of an 
abelian group object if and only if $\wt{A}^2=0$, in which case the abelian
group structure is unique.
\end{prop}
\begin{proof}
An abelian group structure is a map $f\colon A\times_{E_*} A \ra A$ of
$\algapprox$-algebras, which satisfies the axioms for an abelian
group; the unit of the abelian group is necessarily given by the
unique $\algapprox$-algebra map $E_*\ra A$.  Since $A\approx E_*\oplus
\wt{A}$ as $E_*$-modules, we see that the evident
$\algapprox$-algebra map $A\otimes_{E_*} A\ra A\times_{E_*}A$ is
surjective, with kernel $\wt{A}\otimes_{E_*}\wt{A}$.  Thus, a map $f$
satisfying the 
unit axiom for an abelian group  exists if and only if the
multiplication map 
$A\otimes_{E_*} A\ra A$ sends $\wt{A}\otimes_{E_*}\wt{A}$ to $0$.
That is, a 
unital $f$ 
exists if and only if $\wt{A}^2=0$.  It is straightforward to show that
if such a unital  $f$ exists, it is given by $f(c,x,y)=c+x+y$ (written in
terms of  the $E_*$-module
decompositions $A\approx E_*\oplus \wt{A}$ and $A\times_{E_*}A\approx 
E_*\oplus \wt{A}\oplus \wt{A}$), and therefore is the unique abelian group
structure on $A$.
\end{proof}

Thus, the category $(\gralgcat_{/E_*})_{\mathrm{ab}}$ of abelian group
objects is identified with the full subcategory $\gralgcat_{/E_*}$ of
augmented $\algapprox$-algebras with $\wt{A}^2=0$, and the left adjoint to
this inclusion is given by  $A\mapsto A/\wt{A}^2$, which can be regarded as
providing the \dfn{cotangent space} to $A$ at the augmentation $A\ra E_*$.

\subsection{Suspension and loop} 
\label{subsec:suspension-and-loop}

The homotopy category $h(\Alg{E}_{/E})$ of augmented commutative
$E$-algebra spectra admits a \dfn{loop} functor $\Omega\colon
h(\Alg{E}_{/E})\ra h(\Alg{E}_{/E})$.  If $A$ is an augmented
commutative $E$-algebra, then $\Omega A$ is the homotopy pullback 
\[\xymatrix{
{\Omega A} \ar[r] \ar[d]
& {E} \ar[d]
\\
{E} \ar[r]
& {A}
}\]
in $h(\Alg{E}_{/E})$.  The underlying $E$-module spectrum of $\Omega A$
has the form 
$E\vee \Sigma^{-1}\wt{A}$, where $\wt{A}$ is the homotopy fiber of the
augmentation $A\ra E$.  

There is a corresponding loop functor $\Omega\colon \gralgcat_{/E_*}\ra
\gralgcat_{/E_*}$, with the property that as an $E_*$-module $\Omega
A\approx E_*\oplus (E_*S^{-1}\otimes_{E_*} \wt{A})$, where $\wt{A}$ is the
augmentation ideal.  Furthermore, the augmentation ideal of $\Omega A$
will be square-zero, and so $\Omega$ factors through a functor
$\gralgcat_{/E_*}\ra (\gralgcat_{/E_*})_{\mathrm{ab}}$. 

To define the algebraic loop functor, recall the suspension map
$E\colon \Sigma\talgapprox \ra \talgapprox \Sigma$ of
\S\ref{subsec:algebraic-functor}(10).  Since $\Sigma\colon \Mod{E_*}\ra
\Mod{E_*}$ is a self equivalence, it has an inverse functor
$\Sigma^{-1}$, and we may therefore use it to define the \dfn{desuspension
  map} $E'\colon \talgapprox \Sigma^{-1}\ra \Sigma^{-1}\talgapprox$.  

Let $A\ra E_*$ in $\gralgcat_{/E_*}$ with augmentation ideal $\wt{A}$.  As
noted in \S\ref{subsec:augmented-rings}, it is equivalent to consider
$\wt{A}$ as an algebra over $\talgapprox$, with structure map
$\tilde\psi\colon \talgapprox \wt{A}\ra \wt{A}$.  We thus \emph{define} $\Omega
A$ to be the augmented $\algapprox$-algebra with underlying
augmentation ideal $\Sigma^{-1}\wt{A}$, and with structure map
$\tilde{\psi}_{\Omega}$  defined as the composite
\[
\talgapprox \Sigma^{-1}\wt{A} \xra{E'} \Sigma^{-1} \talgapprox \wt{A}
\xra{\Sigma^{-1}\tilde\psi} \wt{A}.
\]
To show that  $\tilde{\psi}_\Omega$ defines a $\talgapprox$-algebra
structure is a straightforward calculation, which relies on the
compatibility of suspension with the monad structure of $\algapprox$
described in \ref{subsec:algebraic-functor}(10).  That the augmentation
ideal of $\Omega A$ is square-zero amounts to
\eqref{prop:diagonal-is-null}.

\subsection{The rings $\Gamma^q$}
\label{subsec:rings-gamma-q}

The ring $\Gamma^q$ is defined to be a ring which naturally acts on
the degree $(-q)$-part of 
the underlying $E_0$-module of a $\algapprox$-algebra, and hence acts
naturally on $\pi_{-q}$ of a $K(n)$-local commutative $E$-algebra.

Let $U^q\colon \gralgcat \ra \Ab$ denote the functor
which sends a 
$\algapprox$-algebra $A$ to its $(-q)$-th grading $A_q$, viewed as an
abelian group.  We \emph{define}  $\Gamma^q\defeq \End(U^q)$, the endomorphism
ring of the functor $U^q$; thus, an element $f\in \Gamma^q$ gives a natural
abelian group homomorphism $A_{-q}\ra A_{-q}$ for all $A$ in
$\gralgcat$. 

The underlying set of $U^q(A)$ is naturally
isomorphic to $\Hom_{\gralgcat}(\algapprox(E_*S^{-q}),A)$, where
$\algapprox (M)$ represents the free $\algapprox$-algebra on an
$E_*$-module $M$.  Therefore, we see that endomorphisms of the
composite functor $\gralgcat\xra{U^q} \Ab\ra \Set$
are exactly the $\algapprox$-algebra endomorphisms of
$\algapprox(E_*S^{-q})$.  Hence, the monoid of set-endomorphisms of
$U^q$ is 
\[
\Hom_{\gralgcat}(\algapprox(E_*S^{-q}), \algapprox(E_*S^{-q})) \approx
\Hom_{\Mod{E_0}}( E_*S^{-q}, \algapprox(E_*S^{-q})) \approx
\algapprox(E_*S^{-q})_{-q},
\]
and so $\Gamma^q$ may be identified with the degree $-q$
\emph{primitives} of
$\algapprox(E_*S^{-q})$.   We thus recover the definition of
$\Gamma^q$ given in 
\cite{rezk-congruence-condition}*{\S7}, where the notation
$\omega^{q/2}$ is used for the $E_*$-module $E_*S^{-q}$.

Using the isomorphisms 
\[
\algapprox_m (E_*S^{-q}) \approx \pi_* L_{K(n)}\Pfree_m(E\sm S^{-q})
\approx \cE{*}B\Sigma_m^{-q\rho_m}
\]
and 
\[\algapprox_m (E_*S^{-q}\oplus E_*S^{-q}) \approx \pi_* L_{K(n)}
\Pfree_m(E\sm (S^{-q}\vee S^{-q})) \approx \bigoplus_{0\leq j\leq
  m}\cE{*} B(\Sigma_j\times \Sigma_{m-j})^{-q\rho_m},
\]
we find that 
\[
\Gamma^q
\approx \bigoplus_{k\geq0}
\Gamma^q[k], 
\]
where
\begin{align*}
\Gamma^q[k] 
& \approx 
\Ker \left[ \algapprox_{p^k}(E_*S^{-q})_{-q} \ra
  \algapprox_{p^k}(E_*S^{-q}\oplus E_*S^{-q})_{-q} \right]
\\
&\approx
\Ker\left[ \cE{-q}B\Sigma_{p^k}^{-q\rho_{p^k}}  \ra
  \bigoplus_{0<j<p^k} \cE{-q}B(\Sigma_j\times
  \Sigma_{p^k-j})^{-q\rho_{p^k}}\right],
\end{align*}
the map in the second line being induced by the transfer map
associated to the inclusion 
$\Sigma_j\times \Sigma_{p^k-j}\subset \Sigma_{p^k}$.

Note that $E_0=\Gamma^q[1]\subseteq \Gamma^q$ is a subring, but is not
central; this is because elements of $\Gamma^q$ given rise to  natural maps of
abelian groups, but not necessarily  to $E_0$-module maps. Thus each
$\Gamma^q[k]$ is naturally an $E_0$-bimodule.  The \emph{left} $E_0$-module
structure on $\Gamma^q[k]$ coincides with the module structure
inherited by the inclusion $\Gamma^q[k]\subseteq
\cE{-q}B\Sigma_{p^k}^{-q\rho_{p^k}}$. 

The following is crucial for understanding the structure of
$\Gamma^q$.  It is proved as
\cite{rezk-congruence-condition}*{Prop.\ 7.3}, though it ultimately
derives from \cite{strickland-morava-e-theory-of-symmetric}.
\begin{prop}\label{prop:gamma-finite-free}
Each $\Gamma^q[k]$ is a finitely generated free  left $E_0$-module.
\end{prop}
\begin{proof}
That these are finitely generated and free (and in fact, $\Gamma^q[k]$
is a direct summand of $\cE{-q}B\Sigma_{p^k}^{-q\rho_{p^k}}$) is proved as
\cite{rezk-congruence-condition}*{Prop.\ 6.1 and Prop.\ 7.2}, an
argument which 
ultimiately derives from
\cite{strickland-morava-e-theory-of-symmetric}.   
\end{proof}
We also know the ranks of the $\Gamma^q[k]$; see
\eqref{prop:gamma-ranks} below.

Here is a variant of the above description of $\Gamma^q$, which we will
need below.  Let
$I^q\colon \gralgcat_{/E_*}\ra \Ab$ denote the functor which sends an
augmented $\algapprox$-algebra to the $(-q)$-degree part of its
augmentation ideal.  There is an evident ring homomorphism
$\End(U^q)\ra \End(I^q)$, and it is straightforward to show that this is
an isomorphism.  That is, $\Gamma^q$ is also the endomorphism ring of
$I^q$.

\subsection{The rings $\Delta^q$}
\label{subsec:rings-delta-q}

The ring $\Delta^q$ is defined to be a ring which acts naturally on
the degree $(-q)$-part of the \emph{cotangent space} to an augmentation  $A\ra
E_*$.

Let $Q^q\colon \gralgcat_{/E_*}\ra \Ab$ denote the functor which
sends an augmented $\algapprox$-algebra $A\ra E_*$ to the $(-q)$-th
grading of its abelianization $\wt{A}/\wt{A}^2$, where $\wt{A}$ is the
augmentation 
ideal.  We \emph{define} 
$\Delta^q\defeq \End(Q^q)$, the endomorphism ring of the functor
$Q^q$; thus an element $f\in 
\Delta^q$ gives a natural abelian group homomorphism
$(\wt{A}/\wt{A}^2)_{-q}\ra 
(\wt{A}/\wt{A}^2)_{-q}$ for all $A$ in $\gralgcat_{/E_*}$.  

To each endomorphism $\phi\colon
Q^q\ra Q^q$ we associate an element $\phi(x)$ in the $E_*$-module 
$(\talgapprox(E_*S^{-q}) /
(\talgapprox(E_*S^{-q}))^2)_{-q}$, defined as the image of
the canonical generator of $E_{-q}S^{-q}$ under the map
\[
x\in E_{-q}S^{-q}\ra
\talgapprox(E_*S^{-q})_{-q} \ra 
Q^q(\algapprox(E_*S^{-q}))\xra{\phi}
Q^q(\algapprox (E_*S^{-q})); 
\]
\begin{prop}
The map $\Delta^q\ra
Q^q(\algapprox(E_*S^{-q})) \approx
(\talgapprox(E_*S^{-q}))/(\talgapprox(E_*S^{-q}))^2)_{-q}$ 
sending $\phi$ to $\phi(x)$  is an isomorphism.
\end{prop}
\begin{proof}
It is straightforward to check, using naturality and the bijection
$\wt{A}_{-q}=\Hom_{\gralgcat/E_*}(\algapprox(E_*S^{-q}),A)$,
that an endomorphism $\phi$ is uniquely determined by the element
$\phi(x)$.  Thus it remains to show that the map of the proposition is
surjective.  

Next, suppose that $\phi$ is an endomorphism of the composite functor
$\gralgcat_{/E_*}\xra{Q^q} \Ab\ra \Set$.  The abelian group structure 
$Q^q(A)\times Q^q(A)\ra Q^q(A)$ is naturally isomorphic to the map
obtained by applying $Q^q$ to the fold map $A\otimes_{E_*} A\ra A$ of
augmented $\algapprox$-algebras.  Thus, by naturality, $\phi$ must
commute with this map.  That is, every set endomorphism of $Q^q$ is
automatically an abelian group endomorphism.  

Given $y\in
Q^q(\algapprox(E_*S^{-q}))$, choose any lift $y\in
\talgapprox(E_*S^{-q})_{-q}$ and consider the corresponding
endomorphism $\psi$ of the composite functor $\gralgcat/E_*\xra{I^q} \Ab\ra
\Set$, which sends an augmented algebra to the underlying set of its
augmentation ideal.  Because $A\mapsto A/\wt{A}^2$ is a functor from
$\gralgcat/E_*$ to itself, we can apply $\psi$ to
$I^q(A/\wt{A}^2)\approx Q^q(A)$ to obtain a natural abelian group
endomorphism of $Q^q$, and a straightforward calculation shows that
the evaluation of this endomorphism on the canonical generator is
exactly $y$, as desired.
\end{proof}

Thus, the $E_0$-module
$\Delta^q$ is isomorphic to the degree $-q$ part of the
indecomposable quotient of the augmented 
$E_*$-algebra: $\algapprox(E_*S^{-q})_{-q} \approx 
\bigoplus_{m\geq0} \cE{-q}B\Sigma_m^{-q\rho_m}$.  The rings
$\Delta^q$ admit a grading $\Delta^q\approx \bigoplus_{k\geq0}
\Delta^q[k]$, where 
\[
\Delta^q[k] \approx \Cok\left[\bigoplus_{0<j<p^k}
  \cE{-q}B(\Sigma_j\times \Sigma_{p^k-j})^{-q\rho_{p^k}}\ra 
\cE{-q}B\Sigma_{p^k}^{-q\rho_{p^k}}\right],
\]
the map being induced by the inclusion $\Sigma_j\times
\Sigma_{p^k-j}\subset \Sigma_{p^k}$.

\subsection{Relation between $\Gamma^q$ and $\Delta^q$}

There are isomorphisms of functors $U^q\approx U^{q+2k}$ and
$Q^q\approx Q^{q+2k}$ for all $k\in \Z$, because $E_*$ is an even
periodic graded ring.  The choice of such isomorphisms depends on a
choice of isomorphism $E_*\approx E_{*+2k}$ of $E_*$-modules.
We obtain the following consequence.
\begin{prop}\label{prop:periodicity-of-power-op-rings}
There are isomorphisms $\Gamma^q\approx \Gamma^{q+2k}$ and
$\Delta^q\approx \Delta^{q+2k}$ of graded rings under $E_0$, for all
$k\in \Z$.
\end{prop}
\begin{proof}
The only thing to note is that the isomorphisms $U^q\approx U^{q+2k}$
and $Q^q\approx Q^{q+2k}$ obtained from the periodicity of $E$
actually  give isomorphisms of
underlying $E_0$-modules, not merely of abelian groups.  Thus the
resulting ring isomorphisms are compatible with the inclusion of the
subring $E_0$.
\end{proof}

Next, we will produce a chain  of ring homomorphisms
\[\xymatrix{
{\Gamma^{q+1}} \ar[d]_{f_{q+1}}
& {\Gamma^q}  \ar[d]_{f_q}
& {\Gamma^{q-1}} \ar[d]^{f{q-1}}
\\
{\Delta^{q+1}} \ar[ur]|{g_{q+1}}
& {\Delta^q}   \ar[ur]|{g_q}
& {\Delta^{q-1}}
}\]
Any natural operation $\phi\colon I^q\ra I^q$ on augmentation ideals,
applied to an 
augmented $\algapprox$-algebra 
$A$, naturally induces an endomorphism $Q^q\ra Q^q$ by passage to the
indecomposables (i.e., evaluate $I^q$ on the quotient map $A\ra
A/\wt{A}^2$, which is a map of $\algapprox$ algebras).  Thus, we
obtain a ring homomorphism $f_q\colon  
\Gamma^q\ra \Delta^q$.  Explicitly, the map $f_q[k]\colon
\Gamma^q[k]\ra \Delta^q[k]$ amounts to the
natural map from the primitive subobject to the indecomposable
quotient of $\cE{-q}B\Sigma_{p^k}^{-q\rho_{p^k}}$. 

\begin{prop}\label{prop:fq-iso}
The map $f_q\colon \Gamma^q\ra \Delta^q$ is an isomorphism if $q$ is odd.
\end{prop}
\begin{proof}
The argument of the proof of \cite{rezk-congruence-condition}*{Prop.\
  7.2} shows that $\algapprox(E_*S^{-q}) \approx \bigoplus
\pi_*L_{K(n)}\Pfree_m \Sigma^{-q}E$ is, as a Hopf algebra, a
primitively generated exterior algebra when $q$ is odd.  Thus, the
evident map from primitives to indecomposables is an isomorphism.  
\end{proof}

Recall that the loop construction (\S\ref{subsec:suspension-and-loop})
gives a functor 
$\Omega\colon \gralgcat_{/E_*}\ra (\gralgcat_{/E_*})_{\mathrm{ab}}\subset
\gralgcat_{/E_*}$.  
Thus, there is a natural isomorphism of abelian groups $Q^q(\Omega
A)\approx 
I^{q-1}(A)$, and hence any endomorphism of $Q^q$ 
induces an endomorphism of $I^{q-1}$.  We have thus defined a ring
homomorphism $g_q \colon \Delta^q\ra \Gamma^{q-1}$.  Explicitly, the
map $g_q[k]\colon \Delta^q[k]\ra \Gamma^{q-1}[k]$ is induced by the
``suspension'' map
\[
\cE{-q} B\Sigma_{p^k}^{-q \rho_{p^k}} \ra \cE{-q+1}
B\Sigma_{p^k}^{(-q+1)\rho_{p^k}},
\]
which factors through the quotient $\Delta^q[k]$ and lands in the
submodule $\Gamma^{q-1}[k]$.  

\begin{prop}
The map $g_q\colon \Delta^q\ra \Gamma^{q-1}$ is an isomorphism for all
$q$.  
\end{prop}
\begin{proof}
In light of \eqref{prop:gamma-finite-free} and \eqref{prop:fq-iso}, it
is enough to show that the 
composite
\[
\Delta^{2q}[k] \xra{g_{2q}} \Gamma^{2q-1}[k] \xra[\sim]{f_{2q-1}}
\Delta^{2q-1}[k] \xra{g_{2q-1}} \Gamma^{2q-2}[k]
\]
is an isomorphism.  Because $E_*$ is $2$-periodic, it is
enough to consider the 
case $q=0$.  Explicitly, we need to show that if we apply
completed $E_*$-homology to the ``zero-section'' map
\[
B\Sigma_{p^k}^+ \ra B\Sigma_{p^k}^{2\bar{\rho}_{p^k}}\approx
\Sigma^{-2} B\Sigma_{p^k}^{2\rho_{p^k}},
\]
the image is exactly the submodule $\Gamma^{-2}[k]$.   (Here
$\bar{\rho}_{p^k}$ denotes the reduced real representation, so that
$\rho_{p^k}=\R\oplus \bar{\rho}_{p^k}$.) 
That this is
the case follows by Theorems 8.5 and 8.6 of
\cite{strickland-morava-e-theory-of-symmetric}, where the result is
stated in ``dual'' form.  Specifically, Strickland proves that
$\mathrm{Prim} 
E^0B\Sigma_{p^k} \ra \mathrm{Ind} E^0B\Sigma_{p^k}$ is generated by
the Euler class of $\bar{\rho}_{p^k}\otimes \C \approx
2\bar{\rho}_{p^k}$, where $\mathrm{Prim}$ is the kernel of
restrictions, and $\mathrm{Ind}$ the quotient of transfers, along
$\Sigma_i\times \Sigma_{p^k-i}\subset \Sigma_{p^k}$.  
\end{proof}

\begin{cor}\label{cor:gamma-delta-iso-free}
All of the rings $\Gamma^q$ and $\Delta^q$ are isomorphic as 
graded rings under $E_0$.  Furthermore, each of the modules $\Gamma^q[k]$ and
$\Delta^q[k]$ are finitely generated free $E_0$-modules.
\end{cor}
\begin{proof}
  By the above, we have isomorphisms $\Delta^0 \xra{g_0} \Gamma^{-1}
  \xra{f_{-1}} \Delta^{-1} \xra{g_{-1}} \Gamma^{-2}$, and the general
  isomorphism 
  follows by \eqref{prop:periodicity-of-power-op-rings}.  The freeness
  follows from \eqref{prop:gamma-finite-free}.  
\end{proof}

Finally, we have the following result about the rank of the free $E_0$-modules
$\Gamma[k]$ (and hence of $\Delta[k]$).
\begin{prop}\label{prop:gamma-ranks}
The ranks of $\Gamma[k]$ and $\Delta[k]$ (as left $E_0$-modules) are
given by the 
generating series 
\[
\sum_k \rank \Gamma[k]\cdot T^k = \bigl[ (1-T)(1-pT)\cdots
(1-p^{n-1}T)\bigr]^{-1},
\]
where $n$ is the height of the formal group associated to $E$.
\end{prop}
\begin{proof}
As noted in the proof of \eqref{prop:gamma-finite-free},
$\Gamma[k]\subseteq \cE{0}B\Sigma_{p^k}$ is the inclusion of a summand
of a free module.  Taking $E_0$-module duals gives a projection 
$E^0B\Sigma_{p^k} \ra E^0B\Sigma_{p^k}/I_{p^k}$ which admits an
$E_0$-module retraction.  Here $I_{p^k}$ is the ideal generated by
transfers from subgroups $\Sigma_r\times \Sigma_{p^k-r}$ for
$0<r<p^k$.  

The $E_0$-rank of $\Gamma[k]$ is equal to the $E_0$-rank of
$E^0B\Sigma_{p^k}/I_{p^k}$.  Strickland
\cite{strickland-morava-e-theory-of-symmetric}*{Thm.\ 1.1} computes
its rank over $E_0$ as
\[
\rank E^0B\Sigma_{p^k}/I_{p^k}  =  \len{\{\text{subgroups
    $A\leq (\Q_p/\Z_p)^n$ with $\len{A}=p^k$}\}}. 
\]
A standard combinatorial argument gives the  generating
series.

\end{proof}

For the remainder of this paper, we write $\Delta$ for the ring
$\Delta^0$; this is the ring we will show is Koszul.

\subsection{Aside:  $(\Gamma^q)^{\sm}_{\mathfrak{m}}$ is the ring of
  additive   operations on the $\pi_*$ of $K(n)$-local $E$-algebras} 
\label{subsec:aside-on-additive-ops}

We here describe the relation between the rings $\Gamma^q$ defined
above, and the rings  of additive operations on homotopy groups of
$K(n)$-local commutative $E$-algebras.

An \dfn{additive operation} in degree $-q$ is a natural endomorphism of
the functor 
\[
\pi_{-q} \colon h\Alg{E,K(n)} \ra \Ab
\]
which computes the $(-q)$th homotopy group of a $K(n)$-local commutative
$E$-algebra.  The ring of additive operations in degree $-q$ is
$\End(\pi_{-q})$ .  
The functor $\pi_{-q}$ (as a functor to sets) is corepresented by the
object 
$L_{K(n)}\Pfree(\Sigma^{-1}E)$, whence 
\[
\End(\pi_{-q})\approx \Ker\bigl[ \pi_{-q}L_{K(n)}\Pfree(\Sigma^{-q}E) \ra
\pi_{-q}L_{K(n)}\Pfree(\Sigma^{-q}E\vee \Sigma^{-q}E)\bigr], 
\]
where the map is induced by the diagonal map $\Sigma^{-q}E\ra
\Sigma^{-q}E\vee \Sigma^{-q}E$.  

\begin{prop}
As rings under $E_0$, $\End(\pi_{-q})\approx
(\Gamma^q)^{\sm}_{\mathfrak{m}}$.  
\end{prop}
\begin{proof}
Recall that by definition $\Gamma^q$ sits in an exact sequence 
\[0\ra \Gamma^q\xra{i}
\algapprox(E_*S^{-q})_{-q}\xra{j} \algapprox(E_*S^{-q}\oplus
E_*S^{-q})_{-q}.\]
We will show that this sequence remains exact after taking
$\mathfrak{m}$-adic completion.
By the isomorphism  of \eqref{subsec:algapprox-properties}(4) and
the exact sequence defining $\End(\pi_{-q})$ given above, the
proposition follows.

To prove the claim, note that the  map $i$ is split (as observed in
the proof of \eqref{prop:gamma-finite-free}, whence it is an
inclusion of a direct sum decomposition
$\algapprox(E_*S^{-1})_{-q}\approx \Gamma^q\oplus N$, and $j$ factors
through an inclusion $N\ra \algapprox(E_*S^{-q}\oplus
E_*S^{-q})_{-q}$.  Recall that $\algapprox$ applied to a free
$E_*$-module is a free module, whence all terms in the sequence, as
well as $N$, are free $E_0$-modules.

Thus, to show that applying $\mathfrak{m}$-adic completion to the
above sequence yields another exact sequence, it suffices to show that
the map $N^{\sm}_\mathfrak{m}\ra (\algapprox(E_*S^{-q}\oplus
E_*S^{-q})_{-q})^\sm_{\mathfrak{m}}$ induced by $j$ is injective.
In fact, if $j\colon N\ra N'$ is any
monomorphism of \emph{free} $E_0$-modules, then the induced map on
$\mathfrak{m}$-adic completions
also a monomorphism.  To see this, simply note that the completion of
$\bigoplus_{i\in I} E_0$ can 
be explictly identified with the set of tuples $(a_i)\in
\prod_{i\in I} E_0$ such that for each $k\geq0$, $a_i\in
\mathfrak{m}^k$ for all but finitely many $i$.

\end{proof}

\section{Koszul rings}
\label{sec:koszul-rings}

In this section, we develop the theory of Koszul rings in terms
of the bar construction (following the original
\cite{priddy-koszul-resolutions}), and in the
generality we need.  Specifically, we describe the theory for a ring
$A$ which contains a commutative ring $R$, but which is \emph{not
  central}  
$A$.
I believe these results are standard, but I do not know a convenient
reference in the literature.  In any case, we need to set up the
interpretation of the Koszul property in terms of the bar
construction, for our results in \S\ref{sec:relation-standard-res}.

Furthermore, we will show that a ring $A$ which is Koszul in our sense
is necessarily \emph{quadratic} \eqref{prop:koszul-implies-quadratic}.
Once we show that the 
ring $\Delta$ is Koszul in our sense, we will have thus proved that it
is quadratic.

In the following let $A=\bigoplus_{k\geq0} A[k]$ be a graded associative ring, and
suppose that $R=A[0]$ is commutative.  It is important that we do
\emph{not} assume that 
$R$ is central in $A$.   We write
$\epsilon\colon A\ra R$ for the evident augmentation map.

\subsection{Bar constructions for rings}
\label{subsec:bar-construction-rings}

Let $M$ be a right $A$-module, and let $N$ be a left $A$-module.
The \dfn{two-sided bar construction} $\B(M,A,N)$ is the simplicial
abelian group
defined by 
$$\B_q(M,A,N)=M\otimes_R\underbrace{A\otimes_R\cdots\otimes_R A}_{\text{$q$
    copies}}\otimes_R N,$$  
with face and boundary maps defined in the usual way.  

Let $\mc{N}\B(M,A,N)$ denote the \emph{normalized} chain complex obtained
from the bar resolution, (obtained by quotienting out by the image of
degeneracy operators); we have
$H_*\mc{N}\B(M,A,N)=H_*\B(M,A,N)$. 

Let $\rB(A)=\B(R,A,R)$, where $R$ is viewed as a left or right
$A$-module using the projection $\epsilon\colon A\ra A[0]\approx R$
defined by $\epsilon(A[m])=0$ for $m>0$.
The complex $\rB(A)$ inherits a grading from the grading on $A$, so
that there is an isomorphism of complexes $\rB(A)\approx \bigoplus
\rB(A)[m]$, where 
$$\rB_q(A)[m] \approx \bigoplus_{m_1+\cdots+m_q=m}
A[m_1]\otimes_R\cdots\otimes_R A[m_q],$$
where each index $m_i>0$.  
Thus the 
homology of $\rB(A)$ is graded, too:
$$H_*\rB(A) \approx \bigoplus_{m\geq0} H_*\rB(A)[m].$$
\begin{prop}
We have that 
$$H_0\rB(A) \approx A[0]=R,$$
and that 
$$H_q\rB(A)[m]=0\quad\text{for $q>m$.}$$
\end{prop}
\begin{proof}
The normalized complex $\mc{N}\rB(A)$ is such that $\mc{N}\rB(A)_q[m]=0$ if
$q>m$, or if $q=0$ and $m>0$.
\end{proof}

Let $\B(A) \defeq \B(A,A,A)$, the ``big'' two-sided bar construction
on $A$.  Since $\epsilon \colon A\ra R$ is a map of $A$-bimodules, there is
an induced surjective map $\B(A)\ra \rB(A)$ of complexes.  Let
$\widetilde{A}=\ker \epsilon$; this is an $A$-bimodule, so we can define
complexes
\[
\hat{\B}(A)\defeq \B(A,A,\widetilde{A}),\quad 
\check{\B}(A)\defeq \B(\widetilde{A},A,A), \quad
\ddot{\B}(A) \defeq \B(\widetilde{A},A,\widetilde{A}),
\]
each of which is naturally a subcomplex of $\B(A)$.  (The notation is
meant to align with the analogous notation for certain subcomplexes of
the nerve of the partition poset, see \eqref{subsec:partition-complex}.)
\begin{prop}\label{prop:exact-bar-complexes}
The sequence of complexes
\[
0\ra \ddot{\B}(A) \xra{(\text{incl.},-\text{incl.})} \hat{\B}(A)
\oplus \check{\B}(A) \xra{(\text{incl.},\text{incl.})} \B(A) \ra 
\rB(A)\ra 0
\]
is exact.
\end{prop}
\begin{proof}
In degree $q\geq0$, $\B(A)_q$ has the form
\[ 
\B(A)_q \approx \bigoplus_{m_0,\dots,m_{q+1}\geq0} A[m_0]\otimes_R
\cdots \otimes_R A[m_{q+1}].
\]
It is straightforward to identify the subgroups $\hat{\B}(A)_q$,
$\check{\B}(A)_q$, 
and $\ddot{\B}(A)_q$ as consisting of those summands with either
$m_{q+1}>0$, 
$m_0>0$, or both, as the case may be.  The result follows easily.
\end{proof}

\subsection{Definition of Koszul rings}
\label{subsec:koszul-def}

We say that the graded ring $A$ is \dfn{Koszul} if 
$$H_q\rB(A)[m]=0\quad \text{for $q<m$.}$$
In view of the previous proposition, this means that if $A$ is Koszul,
then $\rB(A)[m]$ is a chain complex of $R$-modules whose homology is
concentrated in the single dimension $m$.  

\begin{exam}\label{exam:koszul-tensor-algebra}
Let $V$ be an $R$-bimodule.  The \dfn{tensor algebra} $TV$ is defined by
$$TV \defeq \left(\bigoplus_{m\geq0}\underbrace{V\otimes_R\cdots \otimes_R
V}_{\text{$m$ copies}}\right).$$
It is a straightforward exercise to show that $H_q\rB(TV)[m]=0$ for
all $q\geq0$ and all $m>1$.  Thus $TV$ is Koszul.
\end{exam}

For a Koszul ring $A$, let 
\[
C[m] \defeq H_m\rB(A)[m].
\]
\begin{prop}\label{prop:koszul-ranks}
Suppose that $A$ is Koszul, and that each $A[m]$ is finitely generated
and projective as a left $R$-module.  Then each $C[m]$ is finitely
generated and projective as a left $R$-module.  

Suppose each $A[m]$ is as above, and we write
$\rank_{\mathfrak{P}} A[m] = a_m$ for the rank of
$A[m]_{\mathfrak{P}}$ as an $R_{\mathfrak{P}}$-module for some prime
$\mathfrak{P}\subset R$.  Then each
$C[m]$ has
\[
\sum_{m=0}^\infty \rank_{\mathfrak{P}} C[m]\cdot T^m = \left(
  \sum_{m=0}^\infty 
  (-1)^m a_m\cdot T^m
\right)^{-1} 
\]
in $\Z\powser{T}$.
\end{prop}
\begin{proof}
This is straightforward, using the fact that if $P$ and $Q$ are
finitely generated and projective as left $R$-modules, then so is
$P\otimes_R Q$, and that if $R$ is a local ring,  $\rank P\otimes_R Q
= (\rank P)(\rank Q)$. 
\end{proof}

\subsection{Koszul resolutions}
\label{subsec:koszul-resolutions}

Let $M$ be a right, and $N$ a left $R$-module.  Define a filtration
$\{F_m\}$ 
of
the normalized bar complex $\mc{N}\B(M,A,N)$, so that 
\[
F_m = F_m\mc{N}\B_q(M,A,N) \defeq M\otimes_R \left(
  \bigoplus_{m_1+\cdots+m_q\leq m} A[m_1]\otimes_R\cdots \otimes_R
  A[m_q]\right) \otimes_R N.
\]
There are isomorphisms of complexes 
\[
F_m/F_{m-1} \approx M\otimes_R \rB(A)[m] \otimes_R N,
\]
and thus a spectral sequence $E_1^{p,q}=H_p(M\otimes_R \rB(A)[q] \otimes_R N)
\Rightarrow H_{p+q}(\B(M,A,N))$.  This immediately gives the following.
\begin{prop}
Suppose that $A$ is Koszul.
If $M$ and $N$ are flat as right and left $R$-modules respectively,
then the $E_1$-term of the above spectral sequence collapses to a chain complex
of the form
\[
\cdots \ra M\otimes_R C[m] \otimes_R N \ra M\otimes_R C[m-1]\otimes_R
N\ra \cdots.
\]
\end{prop}
This is the \dfn{Koszul complex} for computing $\Tor^A(M,N)$.  Taking
$M=A$ provides for any flat (resp.\ projective) $N$ a flat (resp.\
projective) $A$-module \dfn{Koszul resolution}
\[
\cdots \ra A\otimes_R C[m]\otimes_R N \ra A\otimes_R C[m-1]\otimes_R
N\ra\cdots \ra A\otimes_R C[0]\otimes_R N \ra N\ra 0,
\]
which is an exact sequence of left $A$-modules.

\subsection{Koszul rings are quadratic}

A graded ring $A$ is \emph{quadratic} if it is generated as a ring
under $R$ by
elements of 
degree one, and if all relations are generated by the homogeneous
relations of degree $2$.  
That is, $A$ is quadratic if the natural map
$$T(A[1])/I \ra A$$
is an isomorphism, where $T(V)$ represents the associative ring  over
$R$ freely 
generated by an $R$-bimodule $V$, and
$I$ is a two-sided ideal generated by a sub-bimodule $P\subseteq
A[1]\otimes_R A[1]\subset T(A[1])$. 

\begin{prop}\label{prop:koszul-implies-quadratic}
If $A$ is Koszul, and is flat as a left $R$-module, then $A$ is quadratic.  
\end{prop}
This result is well-known in the case that $R$ is central; see, for
instance, \cite{polishchuk-positselski-quadratic-algebras}.
This will follow from the sharper result
\eqref{prop:quadratic-if-vanishing-h1-and-h2} below.

\begin{prop}
The ring $A$ is generated by $A[1]$ if and only if
$H_1\rB(A)[m]=0$ for all $m>1$.
\end{prop}
\begin{proof}
Since $\rB_0(A)[m]=0$ for $m>0$, we have for each $m\geq2$ an exact
sequence
$$\bigoplus_{\substack{k+\ell=m\\ k,\ell>0}} A[k]\otimes_R A[\ell]
\xra{\text{mult}} A[m]\ra H_1\rB(A)[m]\ra 0.$$ 
Clearly, $H_1\rB(A)[m]=0$ if and only if $A[m]$ is spanned by products
of pairs of elements of strictly lower degree.  The result follows.
\end{proof}

\begin{prop}\label{prop:quadratic-if-vanishing-h1-and-h2}
Suppose $A$ is flat as a left $R$-module.
Then $A$ is quadratic if and only if 
$$H_1\rB(A)[m]=0\quad\text{for all $m>1$,}$$
and
$$H_2\rB(A)[m]=0\quad\text{for all $m>2$}.$$
\end{prop}
\begin{proof}
By the previous proposition, it suffices to show that if $A$ is
generated over $R$ by $A[1]$, then $A$ is quadratic if and only if
$H_2\rB(A)[m]=0$ for all 
$m>2$.  

Let $f\colon
T(A[1])\ra A$ denote the homomorphism of  rings which is the
identity map on $A[1]$; it is surjective and grading 
preserving.  
Consider the resulting exact sequence of chain complexes
$$0\ra K[m]\xra{\gamma} \rB(TA[1])[m] \xra{g} \rB(A)[m]\ra 0,$$
where $g$ is induced by the algebra map $f$.  
Examining the exact sequence of complexes in degrees $1$ and $2$
gives the commutative  diagram
$$\xymatrix@C=10pt@R=20pt{
& {0} \ar[d]
& {0} \ar[d]
\\
{\bigoplus_{\substack{p+q=m \\ p,q>0}} \left(K_1[p]\otimes_R
    A[1]^{\otimes q}\right)
  \oplus \left(A[1]^{\otimes
    p}\otimes_R K_1[q]\right)} \ar@{->>}[r]^-{\beta_m}
\ar[dr]^-{\widetilde{\beta}=\bigoplus(\gamma \otimes \id, 
  \id\otimes \gamma)}
& {K_2[m]} \ar[d] \ar[r]^-{\delta_m}
& {K_1[m]} \ar[d]^{\gamma}
\\
&{\bigoplus_{\substack{p+q=m \\ p,q>0}} A[1]^{\otimes p}\otimes_R A[1]^{\otimes q}}
\ar[d]^{g_2=\bigoplus f\otimes f}
 \ar[r]
& {A[1]^{\otimes m}} \ar[d]^{f}
\\
& {\bigoplus_{\substack{p+q=m \\ p,q>0}} A[p]\otimes_R A[q]} \ar[d] \ar[r]
& {A[m]} \ar[d]
\\
& {0}
& {0}
}$$
in which the columns are exact.  Since $g_2\circ
\widetilde{\beta}=0$, there is a unique lift $\beta_m$ as shown in the diagram,
and $\beta_m$ must be surjective by the flatness hypothesis on $A$.

We have already observed that the complex $\rB(TA[1])[m]$ is acyclic
for all $m>1$ \eqref{exam:koszul-tensor-algebra}.  Thus
$H_2\rB(A)[m]\approx H_1K[m]$ for $m>1$, and 
since $K_0[m]=0$, we obtain an exact sequence 
$$K_2[m] \xra{\delta_m} K_1[m] \ra H_2\rB(A)[m]\ra 0.$$

Putting this all together, we find that (for $m>2$), $H_2\rB(A)[m]=0$ 
if and only if 
$\delta_m$ is surjective, if and only if $\delta_m\beta_m$ is
surjective.   But $\delta_m\beta_m$ being surjective means exactly
that $K_1[m]$ (the module of relations of $A$ of degree $m$) is
generated by relations of lower degree.
\end{proof}

\section{Linearization of functors}
\label{sec:linearization}

In this section, we discuss a certain ``linearization'' operation on
functors between abelian categories.  The linearization construction
gives us a formal approach to the algebra $\Delta$ of
\S\ref{sec:rings-of-power-ops}, and will be used 
later to obtain a bar-resolution of $\Delta$ from partition
complexes. 

\subsection{The linearization construction}
\label{subsec:linearization-def}

Let $\mathcal{A}$ be an additive category, and let $\mathcal{B}$ be an
abelian category.  Let $F\colon \mathcal{A}\ra
\mathcal{B}$ be a (not necessarily additive) functor.  We say $F$ is
\dfn{reduced} if $F0\approx 0$.

For such a reduced functor $F$, we define $\lin_F\colon \mathcal{A}\ra
\mathcal{B}$ together with a natural transformation $\epsilon \colon
F\ra \lin_F$  by setting $\lin_F(X)$ to be the coequalizer of
\[\xymatrix{
{F(X\oplus X)} \ar@<1ex>[rr]^-{F(p_1+p_2)} \ar@<-1ex>[rr]_-{F(p_1)+F(p_2)}
&& {F(X)}
}\]
where $p_i\colon X\oplus X\ra X$ for $i=1,2$ are the two projections.
\begin{prop}
The functor $\lin_F$ is additive; any natural map $F\ra G$ to an
additive functor $G$ factors uniquely through $\epsilon\colon F\ra \lin_F$.
\end{prop}
\begin{proof}
For the first part, note that if $f,g\colon Y\ra X$ are two maps in
$\mathcal{A}$, then 
\[
  \epsilon F(f+g) = \epsilon F(p_1+p_2) F((f,g)) = \epsilon (F(p_1) +
  F(p_2))F((f,g)) = \epsilon(F(f)+F(g)),
\]
as maps $F(Y)\ra \lin_F(X)$, from which it follows that
$\lin_F(f+g)=\lin_F(f)+\lin_F(g)$.   The second part follows from the observation
that $\lin_G\approx G$ when $G$ is additive. 
\end{proof}

Let $\eet F(X) \defeq \ker\left[ F(X\oplus X)
  \xra{(F(p_1),F(p_2))} F(X)\oplus F(X)\right]$, and write
$\beta_F\colon \eet F(X)\ra F(X\oplus X)$ for the inclusion.  Then there is an
exact sequence
\[
\eet F(X) \xra{\gamma_F} F(X) \xra{\epsilon} \lin_F(X) \ra 0.
\]
where $\gamma_F= F(p_1+p_2)\circ \beta_F$.

\subsection{A ``chain rule''}

Given functors $F\colon \mathcal{A}\ra \mathcal{B}$ and $G\colon
\mathcal{B}\ra \mathcal{C}$, where $\mathcal{A}$ is additive and
$\mathcal{B}$ and $\mathcal{C}$ are abelian, there is a
unique 
natural transformation
$c\colon \lin_{F\circ G} \ra  \lin_F\circ \lin_G$ such that $c \epsilon_{F\circ
  G} = \epsilon_F\circ \epsilon_G\colon F\circ G\ra \lin_F\circ \lin_G$,
since $\lin_F\circ \lin_G$ is additive.  Our chain rule says that the
transformation $c$ is an isomorphism whenever things split.

\begin{prop}\label{prop:linearization-composite}
Let $X$ be an object of $\mathcal{A}$.  Suppose that there are direct
sum decompositions 
\[
A\oplus B \xra[\sim]{(i_A,i_B)} G(X), \qquad B\oplus C
\xra[\sim]{(i_{B}', i_C)} \eet G(X),
\]
such that $\epsilon_G i_A\colon A\ra \lin_G(X)$ is an isomorphism,
$\gamma_G i_B' =i_B$, and $\gamma_G i_C=0=\epsilon_G 
i_B$. 
Then the 
natural map $c\colon \lin_{F\circ G}(X)\ra \lin_F(\lin_G(X))$ is an isomorphism.
\end{prop}
\begin{cor}\label{cor:linearization-projective}
Let $F,G\colon \mathcal{A}\ra \mathcal{A}$ be functors on an abelian
category $\mathcal{A}$, and
suppose that $G(X)$ and $\lin_G(X)$ are projective whenever $X$ is a
projective object.  Then $c\colon \lin_{F\circ G}(X) \ra \lin_F(\lin_G(X))$ is
an isomorphism for all projective $X$.
\end{cor}

The proof of \eqref{prop:linearization-composite} is given at the end
of this section.

\subsection{The ring $\Delta$ is the linearization of the monad
  $\talgapprox$}

\begin{prop}\label{prop:linearization-of-algapprox}
Let $m=p^k$, and let $M$ be a free $E_*$-module concentrated in even
degree.  Then the natural map
$\lin_{\algapprox\langle m\rangle}(M) 
\ra \Delta[k]\otimes_{E_0} M$ is an isomorphism.
If $m\neq p^k$, $\lin_{\algapprox \langle m\rangle}(M)=0$.
These maps fit together to give an isomorphism
$\lin_{\talgapprox}(M)\approx \Delta\otimes_{E_0}M$.
\end{prop}
\begin{proof}
Since $\algapprox$ commutes with filtered
colimits (\S\ref{subsec:algebraic-functor}(6)), so does
$\lin_{\algapprox\langle m\rangle}$, and so it is 
enough to consider finite free modules $M=(E_*)^k$. 
Since both $\lin_{\algapprox\langle m\rangle}$ and
$\Delta\otimes_{E_*}({-})$ are additive, it is enough to consider
the module $M=E_*$.   Now we compute that for $m\geq1$, 
\begin{align*}
\eet \algapprox\langle m\rangle (E_*) &= \Ker\left[ 
\algapprox \langle m\rangle(E_*\oplus E_*) \xra{(\algapprox\langle
  m\rangle p_1, \algapprox\langle m\rangle p_2)}  \algapprox \langle
m\rangle (E_*)\oplus \algapprox\langle m\rangle (E_*)\right]
\\
&\approx \bigoplus_{0<i<m} E_*B(\Sigma_i\times \Sigma_{m-i}).
\end{align*}
The map $\algapprox \langle m\rangle (p_1+p_2)\colon \algapprox \langle
m\rangle(E_*\oplus E_*)\ra  \algapprox \langle m\rangle(E_*)$ computes
the effect of $\Pfree_m(p_1+p_2)\colon \Pfree_m(E\vee E)\ra
\Pfree_m(E)$ on homotopy, and we see that $\lin_{\algapprox \langle m
\rangle}(E_*)$ is the cokernel of $\bigoplus_{0<i<m}
E_*B(\Sigma_i\times \Sigma_{m-i})\ra E_*B\Sigma_m$ induced by
transfers.  This map is surjective if $m$ is not a power of $p$; if
$m=p^k$ then in degree $0$ the cokernel is precisely $\Delta[k]$,  and
in general degree the cokernel is $\Delta[k]\otimes_{E_0} E_*$.  
\end{proof}

In particular, it follows that $\lin_{\talgapprox}(M)$ is a free
$E_*$-module whenever $M$ is, since $\Delta$ is a free $E_0$-module
\eqref{cor:gamma-delta-iso-free}.  

\begin{prop}\label{prop:linearization-of-t-bar-is-delta-bar}
There is an isomorphism
\[
\lin_{\B(\talgapprox)}(M) \approx
\B(\Delta)\otimes_{E_0} M
\]
of simplicial $E_*$-modules, which is compatible with grading, so that
$\lin_{\mathcal{B}(\talgapprox)\langle m\rangle }(M) \approx
\mathcal{B}(\Delta)[k]\otimes_{E_0}M$, where $m=p^k$.
\end{prop}
\begin{proof}
If $M$ is a free module, then $\talgapprox(M)$ and
$\lin_{\talgapprox(M)}$ are free.  Thus, the chain rule
\eqref{cor:linearization-projective} applies to show that
$\lin_{\talgapprox^{\circ q}}(M) 
\approx \lin_{\talgapprox}^{\circ q}(M)$ is an isomorphism for
$q\geq0$.  The result follows using \eqref{prop:linearization-of-algapprox}.
\end{proof}

We can make a more refined statement about this isomorphism: all
contributions to $\lin_{\B(\talgapprox)}$ come from the pure weight
part (\S\ref{subsec:graded-exponential}) of 
$\talgapprox^{\circ q}\langle m\rangle$.  
\begin{prop}
\label{prop:linearization-of-t-bar-pure-part}
For all $m=m_1\cdots m_q$ with $m_i=p^{k_i}$ and $k_i\geq0$, the diagram
\[\xymatrix{
{\lin_{\algapprox \langle m_1\rangle \circ\cdots \algapprox\langle
    m_q\rangle} M} \ar[r] \ar[d]
& {\Delta[k_1]\otimes_{E_0}\cdots\otimes_{E_0}\Delta[k_q]
  \otimes_{E_0}M} \ar[d]
\\
{\lin_{\talgapprox^{\circ q}}(M)} \ar[r]
& {\Delta\otimes_{E_0}\cdots \otimes_{E_0}\Delta\otimes_{E_0}M}
}\]
commutes.
\end{prop}
\begin{proof}
This is straightforward from the naturality of the linearization
construction. 
\end{proof}

\subsection{Proof of the chain rule}

We give here a tedious elementary proof of
\eqref{prop:linearization-composite}; afterwards, we indicate how a
somewhat 
more conceptual proof may be constructed, using the results of
\cite{johnson-mccarthy-calculus-cotriples}.

\begin{proof}
First we claim that $c\colon \lin_{F\circ G}(X)\ra \lin_F(\lin_G(X))$ is an
epimorphism.   
The commutative diagram
\[\xymatrix{
{F(A)} \ar[r]^{F(i_A)} \ar[dr]^{\sim}_{F(\epsilon_Gi_A)}
& {F(G(X))} \ar[r]^{\epsilon_{F\circ G}} \ar[d]^{F(\epsilon_G)}
& {\lin_{F\circ G}(X)} \ar[d]^{c}
\\
& {F(\lin_G(X))} \ar[r]_{\epsilon_F}
& {\lin_F(\lin_G(X))}
}\]
shows that $c$ factors through an epimorphism $F(A)\ra \lin_F(\lin_G(X))$,
and thus is an epimorphism.

Let $D = \Ker\left[ F(A\oplus B) \xra{(F(p_1), F(p_2))} F(A)\oplus
  F(B) \right]$, and write $\alpha\colon D\ra F(A\oplus B)$ for the
inclusion.  Let 
$i_D\colon D\ra F(G(X))$ be defined by $F((i_A,i_B))\circ \alpha$.
The commutative diagram
\[\xymatrix{
{F(A)\oplus F(B) \oplus D} \ar[rr]^-{(Fi_A, Fi_B, i_D)}_-{\sim}
\ar[d]_{p_1}
&& {F(G(X))} \ar[r]^{\epsilon_{F\circ G}} \ar[d]_{F\epsilon_G}
& {\lin_{F\circ G}(X)} \ar[d]^{c}
\\
{F(A)} \ar[rr]_{F(\epsilon_Gi_A)}^{\sim}
&& {F(\lin_G(X))} \ar[r]_{\epsilon_F}
& {\lin_F(\lin_G(X))}
}\]
shows that we can identify the kernel of the projection $F(G(X))\ra
\lin_F(\lin_G(X))$ with the image of the bottom horizontal map in 
\[\xymatrix{
&& 
{F(G(X\oplus X))} \ar[d]^{F(G(p_1+p_2))-F(G(p_1))-F(G(p_2))}
\\
{\eet F(A)\oplus F(B)\oplus D} 
\ar[rr]_-{((Fi_A)\gamma_F, Fi_B, i_D)} \ar@{.>}[urr]^{(g_1,g_2,g_3)}
&& {F(G(X))}
}\]
We will construct a section $(g_1,g_2,g_3)$ making the above diagram
commute, and thus prove that the projection $F(G(X))\ra
\lin_F(\lin_G(X))$ factors through an isomorphism $c\colon \lin_{F\circ G}(X)
\ra \lin_F(\lin_G(X))$, as desired.

Let $g_1\colon \eet F(A)\ra F(G(X\oplus X))$ be defined by the composite
\[ 
\eet F(A) \xra{\beta_F} F(A\oplus A) \xra{F(i_A\oplus i_A)}  F(G(X)\oplus
G(X)) \xra{F((G(i_1),G(i_2)))} F(G(X\oplus X)). 
\]
It is straightforward to check that 
\begin{align*}
  F(G(p_\alpha))\circ F((G(i_1),G(i_2))) \circ F(i_A\oplus i_A) &= F(i_A)
  \circ F(p_\alpha),\qquad (\alpha=1,2)
\\ 
  F(G(p_1+p_2)) \circ F((G(i_1),G(i_2))) \circ F(i_A\oplus i_A) &=
  F(i_A) \circ F(p_1+p_2).
\end{align*}
Thus
\[
(F(G(p_1+p_2))-F(G(p_1))-F(G(p_2))) \circ g_1 = F(i_A) \circ \gamma_F
\]
as desired.

Let $g_2\colon F(B)\ra F(G(X\oplus X))$ be defined by the composite
\[
F(B) \xra{F(i_B')} F(\eet G(X)) \xra{F\beta_G} F(G(X\oplus X)).
\]
Since
\[
F(G(p_\alpha))\circ F(\beta_G)=0,\quad (\alpha=1,2),\qquad
F(G(p_1+p_2))\circ F(\beta_G) = F(\gamma_G),
\]
we see that 
\[(F(G(p_1+p_2)-F(G(p_1))-F(G(p_2)))\circ g_2 =
F(\gamma_G)\circ F(i_B')=F(i_B),\]
 as desired.

Let $g_3\colon D\ra F(G(X\oplus X))$ be defined by the composite
\[
D\xra{\alpha} F(A\oplus B) \xra{F(i_A\oplus i_B)} F(G(X)\oplus G(X))
\xra{F((G(i_1), G(i_2)))} F(G(X\oplus X)).
\]
It is straightforward to check that
\begin{align*}
  F(G(p_1))\circ F((G(i_1),G(i_2))) \circ F(i_A\oplus i_B) & =
  F(i_A)\circ F(p_1),
\\ 
  F(G(p_2)) \circ F((G(i_1), G(i_2))) \circ F(i_A\oplus i_B) &=
  F(i_B)\circ F(p_2),
\\
  F(G(p_1+p_2))\circ F((G(i_1),G(i_2))) \circ F(i_A\oplus i_B) &= F((i_A,i_B)).
\end{align*}
Thus
\[
(F(G(p_1+p_2))-F(G(p_1))-F(G(p_2)))\circ g_3 = F(i_D),
\]
as desired.
\end{proof}

\subsection{Proof of the chain rule, using Johnson-McCarthy}

We briefly describe how one may produce a proof of the chain-rule using the
work of \cite{johnson-mccarthy-calculus-cotriples}.  In that paper,
the authors describe a ``derived linearization''
construction which, given a functor $F\colon\mathcal{A}\ra \mathcal{B}$
from an additive category to an abelian category such that $F0=0$,
produces a 
functor $D_1F\colon \mathcal{A}\ra \mathrm{Ch}\mathcal{B}$ to the
category of chain complexes in $\mathcal{B}$.  In degrees $1$ and $0$,
the chain complex $D_1F$ has the form
\[
\cdots \ra\perp F\xra{\gamma_F} F,
\]
and thus $\lin_F=H_0D_1F$.  

According to 
\cite{johnson-mccarthy-calculus-cotriples}*{Lemma 5.7}, there is a
quasi-isomorphism $D_1F\circ D_1G\approx D_1(F\circ G)$, where the
left-hand side is the total complex of the bicomplex obtained by
applying $D_1F$ degreewise to $D_1G$.  Thus, to recover
\eqref{prop:linearization-composite} from their result, it is
necessary to show that $H_0(D_1F\circ D_1G)\approx H_0D_1F\circ
H_0D_1G$.  This is where the hypotheses on $G(X)$ come into play.  
The functor $D_1F$ is additive up to quasi-isomorphism (i.e.,
$D_1F(X)\oplus D_1F(Y)\ra D_1F(X\oplus Y)$ is always a
quasi-isomorphism).  Thus, under the hypotheses of
\eqref{prop:linearization-composite}, applying $D_1F$ to the sequence $\perp
G(X)\ra G(X)$ gives a map quasi-isomorphic to
\[
D_1F(B)\oplus D_1F(C) \xra{\begin{pmatrix} 0 & 0 \\ 1 &
    0 \end{pmatrix}} D_1F(A)\oplus D_1F(B),
\]
and thus $H_0(D_1F\circ D_1G(X))\approx H_0D_1F(A)\approx H_0D_1F\circ
H_0D_1G(X)$, as desired.

\section{Posets and the partition complex}
\label{sec:posets-and-partition-complex}

In this section we describe the \emph{partition complex} of a finite
set together with some other related objects.  These will play a
crucial role in the argument.

\subsection{Posets and their nerves}

First we describe some general notation for simplicial sets associated
to a poset.

Let $X$ be a poset.  The simplicial nerve functor gives a fully faithful
embedding of the category of posets into the category of simplicial
sets.  Thus, it is convenient to regard a poset as
merely a certain kind of simplicial set, and so we use the same
notation of a poset and its simplicial nerve; an \emph{element} of a
poset is exactly a \emph{vertex} of the simplicial set.
If $x_0\leq x_1\leq \cdots \leq x_q$ is a
finite increasing sequence of elements of $X$, we write $[x]=[x_0\leq
x_1\leq \cdots \leq x_q]$ for the corresponding $q$-simplex of the
nerve.
Non-degenerate simplices of the nerve correspond to chains in which
each inequality is strict, which we notate as $[x_0<x_1<\cdots<x_q]$ 

We now suppose that the poset $X$ contains upper and lower bounds, which we
denote $\maxel$ and $\minel$ respectively.  In what follows we will
always assume 
$\maxel\neq \minel$.
We introduce the following
notation:
\begin{enumerate}
\item Let $\hat{X}$ denote the maximal subposet of $X$ which does not include
  $\minel$; we also use this notation to denote the nerve of this poset.
\item Let $\check{X}$ denote the maximal subposet of $X$ which does not
  include $\maxel$; we also use this notation to denote the nerve of
  this poset.
\item Let $\ddot{X}= \hat{X}\cap \check{X}$ as posets; the nerve of
  $\ddot{X}$ is also an intersection of nerves.
\item Let $X^\diamond$ denote the sub-simplicial set of $X$ defined by
  $\hat{X}\cup \check{X}$.
\item Let $\overline{X}$ denote the pointed simplicial set defined by
  $X/X^\diamond$. 
\end{enumerate}
Note that neither $X^\diamond$ nor $\overline{X}$ are generally nerves of
posets.   

Observe that  the simplicial sets $\hat{X}$ and $\check{X}$ are
contractible (to the upper or lower bound) we have that $X^\diamond$
is weakly equivalent to the 
unreduced suspension of $\ddot{X}$, and $\overline{X}$ is weakly equivalent
to the reduced suspension of $X^\diamond$.   

The \dfn{order complex} of a poset $X$, as
studied in algebraic combinatorics, is the simplicial
complex whose 
$q$-simplices are strictly ordered chains $\ol{0}<x_0< x_1< \cdots <
x_{q}<\ol{1}$ of 
elements in $X$ which are not upper or lower bounds.  It corresponds
precisely to the simplicial set 
$\ddot{X}$.

\subsection{Partition complex}
\label{subsec:partition-complex}

Fix $m\geq0$.  Let $P=P_m$ denote the poset of equivalence relations
on the set $\underline{m}=\{1,\dots,m\}$, ordered by refinement.
Thus, a $q$-simplex of $P$ is a chain $[E_0\leq \cdots \leq
E_q]$ of equivalence relations, where we write $E\leq E'$ if $E$ is
``finer'' than $E'$, i.e., if $x\sim_E
y$ implies $x\sim_{E'}y$ for all $x$ and $y$ in the set.

There is an evident action of the symmetric group $\Sigma_m$ on $P_m$,
obtained from the action of $\Sigma_m$ on $\underline{m}$.

We obtain subcomplexes $\hat{P}_m$, $\check{P}_m$, $P_m^\diamond$ of
$P_m$ and a quotient complex $\ol{P}_m=P_m/P_m^\diamond$, all of which
inherit the $\Sigma_m$-action.

\subsection{Pure partitions}
\label{subsec:pure-partition}

Let $m=p^k$ for some $k\geq0$.  A \dfn{pure partition} of a set is
an equivalence relation $E$ such that all the equivalence classes of
$E$ have the same size.  Thus, a partition $E$ of $\underline{m}$ is
pure if and only if all the equivalence classes have order $p^j$,
for some fixed $j\in \{0,\dots,k\}$.  We say that the \dfn{mesh} of such a
pure partition is $j$; we write $\mesh(E)=j$.

The following is elementary but crucial.
\begin{prop}
Let $E$ be a partition of $\underline{m}$, and let $H\leq \Sigma_m$ be
the subgroup consisting of all permutations which fix the partition $E$.
Then $E$ is a pure partition if and only if $H$ acts transitively on
$\ul{m}=\{1,\dots,m\}$. 
\end{prop}

\begin{cor}\label{cor:transitive-fixed-simplex-is-pure}
If $H\leq \Sigma_m$ is a subgroup which acts transitively on
$\ul{m}=\{1,\dots,m\}$, then any $q$-simplex $[E_0\leq \cdots \leq
E_q]$ of $P_m$ which is fixed by $H$ consists of a chain of pure
partitions.  
\end{cor}

\section{The relation between standard resolutions}
\label{sec:relation-standard-res}

This section contains the key observation of this paper: that the bar
resolution of $\Delta$, which we need to understand in order to show
that $\Delta$ is Koszul, is a linearization of the bar resolution of
$\algapprox$, which can be expressed in terms of partition complexes.

\subsection{Transitive $E$-homology}
\label{subsec:transitive-e-homology}

Let $S\subset \underline{2}^{\underline{m}}$ denote the set of
\emph{surjective} functions $\underline{m}\ra \underline{2}$, equipped with
the evident $\Sigma_m$-action; it is a set of order $2^m-2$.

For a spectrum $Z$ equipped with a $\Sigma_m$-action, define 
\[
Q(Z) \defeq \Cok[ \cE{*}(Z\sm \Sip S)_{h\Sigma_m} \ra \cE{*}Z_{h\Sigma_m}],
\]
where the map is the one induced by projection $\pi\colon S\ra *$.
For lack of a better name, we will call this the \dfn{transitive
  $E$-homology} of the $\Sigma_m$-spectrum $Z$.  We extend this
notation to spaces: for a $\Sigma_m$-space $X$, we write $Q(X):=
Q(\Sip X)$, while for a based $\Sigma_m$-space $Y$ we write
$\wt{Q}(Y):= Q(\Si Y)$.  

Note that the functor $Q$ is not a
homology theory; however, it does take
finite coproducts of $\Sigma_m$-spaces to direct sums.
The name comes from the fact that $Q$ only sees
$\Sigma_m$-orbits with transitive isotropy, according to the following
proposition.

\begin{lemma}\label{lemma:transitive-homology-vanishing}
Suppose that $X$ is a finite $\Sigma_m$-set such that for all $x\in
X$, the isotropy group $H\subseteq \Sigma_m$ of $x$ does not act
transitively on $\underline{m}$.  Then $Q(X)=0$.
\end{lemma}
\begin{proof}
Since $Q$ preserves finite coproducts, it suffices to consider
$X=\Sigma_m/H$.  If $H$ does not act transitively on $\underline{m}$,
there exists a surjective function $f\colon \underline{m}\ra
\underline{2}$ which is invariant under the $H$-action, and thus
$\sigma H\mapsto (\sigma H,f\sigma^{-1})\colon X\ra X\times S$ is a
$\Sigma_m$-equivariant section of the projection $X\times S\ra X$.
\end{proof}

We will typically apply $Q$ and 
$\widetilde{Q}$ 
to \emph{discrete} $\Sigma_m$-spaces, obtaining functors
\[
Q\colon G\Set\ra \Mod{E_*},\qquad \wt{Q}\colon G\Set_*\ra
\Mod{E_*}. 
\]

Typically, we will have a
$\Sigma_m$-equivariant simplicial set $X$, and we will apply $Q$ to
each simplicial degree of $X$ separately; thus, if $X$ is a pointed
$\Sigma_m$-equivariant simplicial set, $\widetilde{Q}(X)$ denotes
the simplicial abelian group with
\[
\widetilde{Q}(X)_q \defeq \widetilde{Q}(X_q).
\]

Because the functor $Q$ is actually defined on $\Sigma_m$-spectra, it
is functorial not merely with respect to maps of $\Sigma_m$-sets, but
also with respect to transfers.  That is, 
we
have the following observation. 
\begin{prop}\label{prop:transitive-homology-mackey}
The functor $Q$ is equipped with  the structure of a $\Sigma_m$-Mackey functor.  
\end{prop}
Furthermore, because $Q$ is really defined on $p$-local spectra, we
have the following ``transfer splitting'' result.

\begin{cor}\label{cor:transfer-splitting}
Let $Z$ be a finite $\Sigma_m$-set with order prime to $p$.
Then there exist maps
\[
Q(X) \xra{i} Q(X\times \Sigma_m/H) \xra{j} Q(X),
\]
natural in the $\Sigma_m$-set $X$, such that $ji$ is an isomorphism.
\end{cor}
\begin{proof}
If $Z$ is a finite $G$-set with order prime to $p$, it is standard 
that the composite 
\[
\Sip(*)_{hG} \ra \Sip(Z)_{hG} \ra \Sip(*)_{hG}
\]
of the transfer along $Z\ra *$ followed by the projection along $Z\ra
*$ 
is a $p$-local equivalence.  Applying $Q$ to this sequence gives the
desired result.
\end{proof}

\subsection{Transitive $E$-homology as a linearization}

Fix a space $W$ equipped with a $\Sigma_m$-action.
Let $F\colon \hSpectra\ra \Mod{E_*}$ be the functor defined by
\[
F(Y) \defeq \cE{*}(\Sip W\sm Y^{\sm m})_{h\Sigma_m}.
\]
Let $\lin_F\colon \hSpectra \ra \Mod{E_*}$ be the linearization of
$F$, as described in \S\ref{sec:linearization}.
\begin{prop}\label{prop:transitive-homology-linearization}
For all spaces $X$ there is a natural isomorphism
\[
\lin_F(\Sip X) \ra Q(W\times X^m).
\]
\end{prop}
\begin{proof}
The sequence
\[
Y^{\sm m}\sm \Sip S \ra (Y\vee Y)^{\sm m} \xra{(p_1,p_2)} Y^{\sm
  m}\vee Y^{\sm m}
\]
is a split cofibration sequence in the category of $\Sigma_m$-spectra,
from which it follows that the map $\gamma_F\colon \perp F(Y)\ra F(Y)$ is
isomorphic to the map $\cE{*}(\Sip W\sm Y^{\sm m}\sm \Sip
S)_{h\Sigma_m}\ra 
\cE{*}(\Sip W\sm Y^{\sm m})_{h\Sigma_m}$ induced by projection $S\ra *$.
\end{proof}

\subsection{Transitive $E$-homology of the partition complex}

We have the following.
\begin{prop}\label{prop:transfer-partition}
Let $P=P_m$ be the nerve of the partition poset.  Then the sequence of
simplicial abelian groups
\[
0\ra Q(\ddot{P}) \xra{(\text{incl.},-\text{incl.})} Q(\hat{P})\oplus
Q(\check{P}) \xra{(\text{incl.},\text{incl.})} Q(P) \ra 
\widetilde{Q}(\overline{P}) \ra 0
\]
is exact.
\end{prop}
\begin{proof}
For $q\geq0$, let $P_q$ denote the $q$-simplices of $P$.  Then
\[
P_q \approx (P_q-P_q^\diamond)\amalg (\hat{P}_q-\ddot{P}_q)\amalg
(\check{P}_q-\ddot{P}_q) 
\amalg \ddot{P}_q
\]
as $\Sigma_m$-sets, and the result follows because $Q$ preserves
coproducts.
\end{proof}

\subsection{The partition complex and $E_\infty$-operads}

A \dfn{symmetric sequence} is a functor $A\colon \Sigma\ra \Spaces$,
where, $\Sigma$ denotes the groupoid of finite sets and isomorphisms.
A symmetric sequence $A$ determines a functor 
\[
\mathcal{C}_A\colon \Spaces\ra \Spaces\quad\text{by}\quad
\mathcal{C}_A(X) \defeq \coprod_{m\geq0} (A(\underline{m})\times
X^m)_{\Sigma_m},
\]
and the assignment $A\mapsto \mathcal{C}_A$ is functorial.
There is a monoidal product $A,B\mapsto A\circ B$, which has the
property that $\mathcal{C}_{A\circ B}\approx
\mathcal{C}_A\mathcal{C}_B$.
This monoidal product satisfies the formula
\[
(A\circ B)(S) \approx \coprod_{n,\; f\colon S\ra \underline{n}}
A(\underline{n})\times \prod_{s\in S} B(f^{-1}(s))
\]
where the coproduct runs over integers $n\geq0$ and functions $f\colon
S\ra \underline{n}$.   An operad is a monoid with respect to this
monoidal product.

If $O$ is an operad, then
$\B(O)=\B(O,O,O)$ is a simplicial object in
symmetric sequences.  
\begin{prop}\label{prop:iso-bar-c-with-partition-general}
Let $O$ be the non-unital $E_\infty$-operad in spaces. 
For each $m\geq0$ there is a map $\B(O)(\ul{m})\ra P_m$ of simplicial
$\Sigma_m$-spaces, which is a weak equivalence of spaces in each
simplicial degree.  That is, $O^{\circ (q+2)}(\ul{m}) \ra
(P_m)_q$ is a weak 
equivalence. 
\end{prop}
\begin{proof}
  A standard and well-known combinatorial argument.  In fact, taking
  $O$ to be the non-unital commutative operad, with $O(S)$ a one-point
  space for all non-empty $S$, gives an isomorphism
  $\B(O)(\ul{m})\approx P_m$ of simplicial $\Sigma_m$-spaces. 
\end{proof}

We may consider the monad $\wt{C}=C_O$ associated to
$O$, together 
with its associated monadic bar construction
$\B(\wt{C})=\B(\wt{C},\wt{C},\wt{C})$.
Evaluating at a space $X$ gives a simplicial
space $\B(\wt{C})(X)=
\B(\wt{C},\wt{C},\wt{C}(X))$, and applying
\eqref{prop:iso-bar-c-with-partition-general} leads to the following.
\begin{prop}\label{prop:iso-bar-c-with-partition}
There is an isomorphism
\[
\B(\wt{C})(X) \approx \coprod_{m\geq 0} \left(P_m\times X^m\right)_{h\Sigma_m},
\]
natural in CW-complexes $X$, 
in the category of simplicial objects in the homotopy category of
spaces.  In particular, for each $q\geq 0$ there are natural weak
equivalences 
\[
\wt{C}^{\circ (q+2)}(X) \approx \coprod_{m\geq0}\left((P_m)_q \times
  X^m\right)_{h\Sigma_m}. 
\]
\end{prop}
Observe that there is a coproduct decomposition
\[
\B(\wt{C})\approx \coprod_{m\geq1} \B(\wt{C})\langle m\rangle
\]
in the category of simplicial functors, so that there are natural weak
equivalences 
\[
\B(\wt{C})\langle m\rangle(X) \approx (P_m\times X^m)_{h\Sigma_m}.
\]

\subsection{A fundamental observation}
\label{subsec:fundamental-obs}

Our approach to proving that $\Delta$ is a Koszul ring comes from the
following observation: if we apply $E$-homology to $\B(\wt{C})(X)$ in each
simplicial degree, and then ``linearize''  with respect
to $X$, and 
set $X=*$, then what we obtain is the bar complex $\B(\Delta)$ for the
ring $\Delta$.  This linearization, in turn, turns out to be exactly
the transitive homology of the partition complex.

\begin{prop}\label{prop:fundamental-observation-monad-koszul}
For $m=p^k$, there is an isomorphism of simplicial abelian groups
\[
\B(\Delta)[k] \ra Q(P_m).
\]
Furthermore, this isomorphism carries the subobjects
$\hat{\B}(\Delta)[k]$, $\check{\B}(\Delta)[k]$, and
$\ddot{\B}(\Delta)[k]$ isomorphically to the subobjects
$Q(\hat{P}_m)$, $Q(\check{P}_m)$, and $Q(\ddot{P}_m)$.
\end{prop}

Note that if $m$ is not a power of $p$, then
$\widetilde{Q}(\overline{P}_m)=0$.  We give the proof of
\eqref{prop:fundamental-observation-monad-koszul} below.

\begin{cor}\label{cor:fundamental-observation-cor}
For $m=p^k$, $k\geq0$, there is an isomorphism of simplicial abelian groups
\[
\rB(\Delta)[k] \approx \widetilde{Q}(\overline{P}_{m}).
\]
\end{cor}
\begin{proof}
Immediate using \eqref{prop:fundamental-observation-monad-koszul} to
compare the exact sequences of \eqref{prop:exact-bar-complexes} and 
\eqref{prop:transfer-partition}. 
\end{proof}

Now we relate the simplicial object $\B(\wt{C})(X)$ with the
non-unital version $\talgapprox$ of the algebraic approximation
functor of \S\ref{subsec:augmented-rings}.

\begin{prop}\label{prop:algapprox-bar-is-c-bar}
If $X$ is a finite $\Sigma_m$-set, $m=p^k$, then there is an isomorphism
\[
\B(\talgapprox)\langle m\rangle(E_*X) \xra{\alpha}
\cE{*}\B(\wt{C})\langle m\rangle(X)
\]
of simplicial $E_*$-modules.
\end{prop}
\begin{proof}
This amounts to the fact that
\[
(\talgapprox^{\circ q})\langle m\rangle (E_*X) \approx \cE{*} (O^{\circ q}(\underline{m}) \times
X^m)_{h\Sigma_m},
\]
where $O$ is a non-unital $E_\infty$-operad, using the fact that the
algebraic approximation functor is an isomorphism on finitely
generated free $E_*$-modules such as $E_*X$
\eqref{subsec:algebraic-functor}(3). 
\end{proof}

\begin{proof}[Proof of
  \eqref{prop:fundamental-observation-monad-koszul}]
By \eqref{prop:algapprox-bar-is-c-bar} and
\eqref{prop:iso-bar-c-with-partition}, there are isomorphisms
\[
\B(\talgapprox)\langle m\rangle(E_*X)\approx
\cE{*}\B(\wt{C})\langle m\rangle(X)\approx
\cE{*}(P_m\times X^m)_{h\Sigma_m}
\]
of simplicial $E_*$-modules, natural in $\Sigma_m$-sets $X$.
Applying \eqref{prop:transitive-homology-linearization}, we see that
\[
\lin_{\B(\talgapprox)\langle m\rangle}(\Sip X) \approx Q(P_m\times X^m).
\]
The isomorphism $\B(\Delta)[k]\approx Q(P_m)$ follows from the isomorphism 
\[
\lin_{\B(\algapprox)}(M)\approx \B(\Delta)\otimes_{E_0}M
\]
of \eqref{prop:linearization-of-t-bar-is-delta-bar}.

To see that each of these isomorphisms induces isomorphisms
of the relevant subobjects
\[
\hatB(\Delta)[k]\approx Q(\hat{P}_m),\quad \chB(\Delta)[k]\approx
Q(\check{P}_m),\quad \ddB(\Delta)[k]\approx Q(\ddot{P}_m),
\]
we recall from \eqref{prop:linearization-of-t-bar-pure-part} that the
summands $\Delta[k_{q+1}]\otimes_{E_0} \cdots
\otimes_{E_0}\Delta[k_{0}]$ of 
$\Delta^{\otimes (q+2)}$ come from the linearization of the pure part
$\algapprox \langle m_{q+1}\rangle\circ\cdots\circ \algapprox \langle
m_{0}\rangle$ of $\talgapprox^{\circ (q+2)}$, where $m_i=p^{k_i}$.

On the other hand, since $Q$ vanishes on $\Sigma_m$-orbits whose
isotropy does not act transitively on $\ul{m}$
\eqref{lemma:transitive-homology-vanishing},  we see from
\eqref{cor:transitive-fixed-simplex-is-pure} that
$Q(P_m)$ depends only on the the pure part of the
partition complex; i.e., $Q(P_{m,q})=Q(P_{m,q}^{\mr{pure}}$) where
$P_{m,q}$ is the set of $q$-simplices of $P_m$ and
$P_{m,q}^{\mr{pure}}\subseteq P_{m,q}$ is the subset consisting of
chains of pure partitions.

Thus  $\B_q(\Delta)[k]$
  is actually isomorphic to $Q$ of the subset of $P_{m,q}$ consisting of chains
  $[E_0\leq\cdots \leq E_q]$ consisting of pure partitions with
  $\mesh(E_k)=m_0+\cdots+ m_k$.  Tracing through the isomorphism, we see
  that the summand
$\hatB_q(\Delta)[k]$ corresponds to the subset of chains of pure
partitions with $m_0>0$, while  $\chB_q(\Delta)[k]$ corresponds to the
subset of chains of pure partitions with $m_q>0$.  The claim follows.
\end{proof}

\section{Bredon homology of the partition complex}

In this section we complete the proof that the ring $\Gamma$ of power
operations for Morava $E$-theory is Koszul, by
applying a result of Arone, Dwyer, and Lesh
\cite{arone-dwyer-lesh-bredon-partition}, by proving that the
simplicial abelian group $\wt{Q}(\ol{P}_m)$ has homology only in
degree $m$.

\begin{rem}
The original version of this paper contained a extensive proof of this
fact, which can be viewed as an elaborate generalization of the
argument of Solomon and Tits \cite{solomon-steinberg-character} on the
homotopy type of the Tits building of a BN-pair, as realized in the
case of the group $G=GL(k,\mathbb{F}_p)$ with its usual BN-structure.
The author is rather fond of the original proof, but there is no
reason to reproduce it here in light of the Arone-Dwyer-Lesh result.
It can be found in the first (version 1) arXiv posting of this paper.  
\end{rem}

\subsection{Bredon homology of partition complexes}

We first identify a special case of the theorem of
Arone-Dywer-Lesh, which we will apply to our calculation. 
\begin{prop}[Special case of
  \cite{arone-dwyer-lesh-bredon-partition}*{Cor.\ 1.2}]
\label{prop:adl-special-case}
Let $m=p^k$ for some $k\geq 1$, and let  $M$  be a Mackey functor for
$G=\Sigma_m$ taking values 
in  $\Z_{(p)}$-modules, 
which satisfies the following hypotheses.
\begin{enumerate}
\item [(i)] For all finite $G$-sets $X$ and $Z$ such that the order of
  $Z$ is prime to $p$, the composite
\[
M(X) \ra M(X\times Z)  \ra M(X)
\] 
is an isomorphism, where  the maps in question are respectively the
contravariant and covariant maps induced by the projection $X\times
Z\ra X$ of finite $G$-sets.
\item [(ii)] If $H\leq G$ is a subgroup which acts non-transitively on
  $\ul{m}=\{1,\dots, m\}$, then $M(G/H)\approx 0$.
\end{enumerate}
Then there are isomorphisms
\[
\wt{H}^{\Sigma_m}_j(P_m^\diamond; M) \approx 
\begin{cases}
  0 & \text{if $j\neq k-1$,}
\\
  M(\Sigma_m/V) \otimes_R \mr{St}_k & \text{if $j=k-1$,}
\end{cases}
\]
where $V\leq \Sigma_m$ is a subgroup isomorphic to $(\Z/p)^k$,
$R=\Z_{(p)}[\operatorname{Aut}(V)]=\Z_{(p)}[GL_k(\mathbb{F}_p)]$, and
$\mr{St}_k$ is the Steinberg representation of $GL_k(\mathbb{F}_p)$.  
\end{prop}

Recall the transitive $E$-homology functor $Q$ of the previous
section.  Restricted to finite $G$-sets, it defines a Mackey functor,
which we also denote by $Q$.  
\begin{prop}\label{prop:iso-to-bredon}
For a based simplicial  set $X$ with $G=\Sigma_m$ action, the homology
of the simplicial 
abelian group $\wt{Q}(X)$ is precisely
\emph{(reduced) Bredon homology} of $X$ with coefficients in $Q$:
\[
H_*\wt{Q}(X) \approx \wt{H}^{\Sigma_m}_*(X;Q).
\]
\end{prop}
\begin{proof}
This is simply the definition of Bredon homology of a $G$-simplicial set.
\end{proof}

We obtain the result we need as a corollary of
\eqref{prop:adl-special-case}.
\begin{cor}\label{cor:bredon-vanishing}
For $m=p^k$ and $k\geq1$ we have that
\[
H_j\wt{Q}(X)= \wt{H}^{\Sigma_m}_j(\ol{P}_m;Q) =0\qquad \text{if $j\neq k$.}
\]
\end{cor}
\begin{proof}
As $\ol{P}_m=P_m/P_m^\diamond$ and $P_m$ is equivariant contractible,
we see that $\ol{P}_m$ is equivalent to the reduced suspension of
$P_m^\diamond$, so $H_j^{\Sigma_m}(\ol{P}_m; Q)=
H_{j-1}^{\Sigma_m}(P_m^\diamond; Q)$.  The claim then follows from
\eqref{prop:adl-special-case} once we see that $Q$ satisfies (i) and
(ii).  

For condition (i),  recall that $Q$ is defined as a quotient of  the
Mackey functor $X\mapsto \cE{*}(X_{\Sigma_m})$, which also satisfies
property (i), since the composite $X_{h\Sigma_m}\ra (X\times
Z)_{h\Sigma_m}\ra X_{h\Sigma_m}$ of transfer followed by projection is
a $p$-local equivalence when the order of $Z$ is prime to $p$.  

Condition (ii) is satisfied by $Q$ by construction.
\end{proof}

Next we derive our special case from the theorem stated in
\cite{arone-dwyer-lesh-bredon-partition}.  
\begin{proof}[Proof of \eqref{prop:adl-special-case}]
This is a special case of
\cite{arone-dwyer-lesh-bredon-partition}*{Cor.\ 1.2}, whose conclusion
is precisely that stated in \eqref{prop:adl-special-case}, depending on
hypotheses on a $Z_{(p)}$-Mackey functor $M$ stated as (1)--(3) of
\cite{arone-dwyer-lesh-bredon-partition}*{Cor.\ 1.2}.  They prove that
our condition
(i) implies their condition (1); this is stated as
\cite{arone-dwyer-lesh-bredon-partition}*{Lemma 3.8}.  Their
 (2) and (3) are conditions on the values of $M(\Sigma_m/D)$ where
$D\leq \Sigma_m$ is an elementary abelian subgroup which acts freely
and non-transitively on $\ul{m}=\{1,\dots,m\}$.  These conditions (2)
and (3) are trivially implied by our (ii), since that implies
$M(\Sigma_m/D)=0$.  
\end{proof}

\begin{rem}
Kathryn Lesh has pointed out  that the special case
\eqref{prop:adl-special-case} really 
is the ``easy case'' of the machinery of
\cite{arone-dwyer-lesh-bredon-partition}.  In particular, the proof
of the special case we use can be extracted entirely from sections 1--6 and
10 of their paper (the proof of the main theorem being completed in
section 10).  The key point is the hypothesis (in
\cite{arone-dwyer-lesh-bredon-partition}*{Prop.\ 5.4}) that certain
collections of subgroups are ``homologically $\mu(G/D)$-ample'', which
is a vanishing condition on the homology of a certain space with
$G$-action, with twisted coefficients $\mu(G/D)$.  

In the case relevant to the proof of their theorem, $\mu(G/D)=M(G/D)$
with $D\leq G=\Sigma_m$ a non-transitive elementary $p$ group.
Our condition (ii) ensures that $M(G/D)=0$ and thus this hypothesis of
their Prop.\ 5.4 is trivially satisfied.  Sections 7--9 of their paper
do the hard work of showing that their sharper conditions (2) and (3)
given rise to the needed hypothesis; these aren't needed if we only
need the special case we use.

\end{rem}

\subsection{Proof of the main theorem}
\label{subsec:proof-of-main-thm}

Now we can give the proof of our main theorem.
\begin{proof}[Proof of \eqref{main-theorem}]
Recall that the graded ring $\Gamma$ is isomorphic to the graded ring
$\Delta$ \eqref{cor:gamma-delta-iso-free}, so it suffices to show that
$\Delta$ is Koszul.  Recall
\eqref{subsec:koszul-def} that $\Delta$ is Koszul  if
$H_q(\rB(\Delta)[m])=0$ for $q<m$.  By
\eqref{cor:fundamental-observation-cor} there is are isomorphisms of
simplicial abelian groups $\rB(\Delta)[k]\approx \wt{Q}(\ol{P}_m)$.
By \eqref{prop:iso-to-bredon} the homology of this simplicial abelian
group is the Bredon homology $\wt{H}_*^{\Sigma_m}(\ol{P}_m; Q)$ with
coefficients in transitive homology Mackey functor $Q$, and the
desired vanishing is \eqref{cor:bredon-vanishing}. 
\end{proof}

Finally, we prove that the Koszul resolution
\eqref{subsec:koszul-resolutions}  associated to $\Gamma$ (or
$\Delta$) has
length $n+1$, where $n$ is the height of the formal group associated
to $E$.
\begin{prop}\label{prop:koszul-vanishing}
For a height $n$-Morava $E$-theory, we have that $C[k]:=
H_k\rB(\Delta)[k] \approx H_k^{\Sigma_{p^k}}(\ol{P}_{p^k};Q)\approx
0$, and thus any $\Gamma$-module $N$ which is flat (resp.\ projective)
over $E_0$ admits a flat (resp.\ projective) $\Gamma$-module
resolution of length $n+1$.
\end{prop}
\begin{proof}
Recall \eqref{prop:koszul-ranks} that since $\Delta$ is Koszul,  each
$C[k]$ is finitely generated and projective as a left $E_0$-module
(and thus is free since $E_0$ is a local ring).  We have
\begin{align*}
  \sum_{k=0}^\infty \rank C[k]\cdot T^k 
&= \bigl( \sum_{k=0}^\infty (-1)^k (\rank \Delta[k])\cdot
  T^k\bigr)^{-1} 
& \text{by \eqref{prop:koszul-ranks}}
\\
&= \bigl(\sum_{k=0}^\infty \rank \Gamma[k] \cdot (-T)^k \bigr)^{-1}
& \text{using $\Gamma[k]\approx \Delta[k]$ \eqref{cor:gamma-delta-iso-free}}
\\
&= (1+T)(1+pT)\cdots (1+p^{n-1}T)
& \text{by \eqref{prop:gamma-ranks},}
\end{align*}
whence $\rank C[k]=0$ if $k>n$.  
\end{proof}

\begin{rem}
This actually  shows the ranks of the $C[k]$ are ``Gaussian binomial
coefficients'': 
\[
\rank C[k] = \len{\{\text{subspaces $V\subset \F_p^n$ with $\dim V =
    k$}\}}. 
\]
\end{rem}



\begin{bibdiv}
\begin{biblist}
\bib{ando-isogenies-and-power-ops}{article}{
  author={Ando, Matthew},
  title={Isogenies of formal group laws and power operations in the cohomology theories {$E\sb n$}},
  journal={Duke Math. J.},
  volume={79},
  date={1995},
  number={2},
  pages={423--485},
  issn={0012-7094},
}

\bib{ando-hopkins-strickland-h-infinity}{article}{
  author={Ando, Matthew},
  author={Hopkins, Michael J.},
  author={Strickland, Neil P.},
  title={The sigma orientation is an $H\sb \infty $ map},
  journal={Amer. J. Math.},
  volume={126},
  date={2004},
  number={2},
  pages={247--334},
  issn={0002-9327},
}

\bib{arone-dwyer-lesh-bredon-partition}{article}{
  author={Arone, G. Z.},
  author={Dwyer, W. G.},
  author={Lesh, K.},
  title={Bredon homology of partition complexes},
  journal={Doc. Math.},
  volume={21},
  date={2016},
  pages={1227--1268},
  issn={1431-0635},
}

\bib{arone-mahowald-identity-functor}{article}{
  author={Arone, Greg},
  author={Mahowald, Mark},
  title={The Goodwillie tower of the identity functor and the unstable periodic homotopy of spheres},
  journal={Invent. Math.},
  volume={135},
  date={1999},
  number={3},
  pages={743--788},
  issn={0020-9910},
}

\bib{barthel-frankland-completed-approximation}{article}{
  author={Barthel, Tobias},
  author={Frankland, Martin},
  title={Completed power operations for Morava $E$-theory},
  journal={Algebr. Geom. Topol.},
  volume={15},
  date={2015},
  number={4},
  pages={2065--2131},
  issn={1472-2747},
}

\bib{bmms-h-infinity-ring-spectra}{book}{
  author={Bruner, R. R.},
  author={May, J. P.},
  author={McClure, J. E.},
  author={Steinberger, M.},
  title={$H\sb \infty $ ring spectra and their applications},
  series={Lecture Notes in Mathematics},
  volume={1176},
  publisher={Springer-Verlag},
  place={Berlin},
  date={1986},
  pages={viii+388},
  isbn={3-540-16434-0},
}

\bib{goerss-hopkins-moduli-spaces}{article}{
  author={Goerss, P. G.},
  author={Hopkins, M. J.},
  title={Moduli spaces of commutative ring spectra},
  conference={ title={Structured ring spectra}, },
  book={ series={London Math. Soc. Lecture Note Ser.}, volume={315}, publisher={Cambridge Univ. Press}, place={Cambridge}, },
  date={2004},
  pages={151--200},
}

\bib{johnson-mccarthy-calculus-cotriples}{article}{
  author={Johnson, B.},
  author={McCarthy, R.},
  title={Deriving calculus with cotriples},
  journal={Trans. Amer. Math. Soc.},
  volume={356},
  date={2004},
  number={2},
  pages={757--803 (electronic)},
  issn={0002-9947},
}

\bib{kashiwabara-k2-homology}{article}{
  author={Kashiwabara, Takuji},
  title={$K(2)$-homology of some infinite loop spaces},
  journal={Math. Z.},
  volume={218},
  date={1995},
  number={4},
  pages={503--518},
  issn={0025-5874},
}

\bib{may-general-algebraic-approach}{article}{
  author={May, J. Peter},
  title={A general algebraic approach to Steenrod operations},
  conference={ title={The Steenrod Algebra and its Applications (Proc. Conf. to Celebrate N. E. Steenrod's Sixtieth Birthday, Battelle Memorial Inst., Columbus, Ohio, 1970)}, },
  book={ series={Lecture Notes in Mathematics, Vol. 168}, publisher={Springer, Berlin}, },
  date={1970},
  pages={153--231},
}

\bib{mcclure-dyer-lashof-k-theory}{article}{
  author={McClure, James E.},
  title={Dyer-Lashof operations in $K$-theory},
  journal={Bull. Amer. Math. Soc. (N.S.)},
  volume={8},
  date={1983},
  number={1},
  pages={67--72},
  issn={0273-0979},
}

\bib{polishchuk-positselski-quadratic-algebras}{book}{
  author={Polishchuk, Alexander},
  author={Positselski, Leonid},
  title={Quadratic algebras},
  series={University Lecture Series},
  volume={37},
  publisher={American Mathematical Society},
  place={Providence, RI},
  date={2005},
  pages={xii+159},
  isbn={0-8218-3834-2},
}

\bib{priddy-koszul-resolutions}{article}{
  author={Priddy, Stewart B.},
  title={Koszul resolutions},
  journal={Trans. Amer. Math. Soc.},
  volume={152},
  date={1970},
  pages={39--60},
  issn={0002-9947},
}

\bib{rezk-dyer-lashof-example}{article}{
  author={Rezk, Charles},
  title={Power operations for Morava $E$-theory of height $2$ at the prime $2$},
  date={2008},
  eprint={arXiv:0812.1320 (math.AT)},
}

\bib{rezk-congruence-condition}{article}{
  author={Rezk, Charles},
  title={The congrugence criterion for power operations in Morava $E$-theory},
  journal={Homology, Homotopy Appl.},
  volume={11},
  date={2009},
  number={2},
  pages={327--379},
  issn={1532-0073},
  eprint={arXiv:0902.2499},
}

\bib{rezk-modular-complexes}{article}{
  author={Rezk, Charles},
  title={Modular isogeny complexes},
  journal={Algebraic \& Geometric Topology},
  volume={12},
  date={2012},
  number={3},
  pages={1373--1403},
  eprint={:arXiv:1102.5022},
}

\bib{rezk-power-ops-structure-calculation}{article}{
  author={Rezk, Charles},
  title={Power operations in Morava $E$-theory: structure and calculations},
  date={2013},
  eprint={http://www.math.uiuc.edu/~rezk/power-ops-ht-2.pdf},
}

\bib{solomon-steinberg-character}{article}{
  author={Solomon, Louis},
  title={The Steinberg character of a finite group with $BN$-pair},
  conference={ title={Theory of Finite Groups (Symposium, Harvard Univ., Cambridge, Mass., 1968)}, },
  book={ publisher={Benjamin, New York}, },
  date={1969},
  pages={213--221},
}

\bib{strickland-finite-subgroups-of-formal-groups}{article}{
  author={Strickland, Neil P.},
  title={Finite subgroups of formal groups},
  journal={J. Pure Appl. Algebra},
  volume={121},
  date={1997},
  number={2},
  pages={161--208},
  issn={0022-4049},
}

\bib{strickland-morava-e-theory-of-symmetric}{article}{
  author={Strickland, N. P.},
  title={Morava $E$-theory of symmetric groups},
  journal={Topology},
  volume={37},
  date={1998},
  number={4},
  pages={757--779},
  issn={0040-9383},
  eprint={arXiv:math/9801125},
}

\bib{zhu-power-ops-at-3}{article}{
  author={Zhu, Yifei},
  title={The power operation structure on Morava $E$-theory of height 2 at the prime 3},
  journal={Algebr. Geom. Topol.},
  volume={14},
  date={2014},
  number={2},
  pages={953--977},
  issn={1472-2747},
}

\bib{zhu-modular-equations-lubin-tate}{article}{
  author={Zhu, Yifei},
  title={Modular equations for Lubin-Tate formal groups at chromatic level 2},
  date={2015},
  eprint={arXiv:1508.03358},
}


\end{biblist}
\end{bibdiv}
\end{document}